\documentclass[12pt,leqno]{amsart}
\usepackage[latin1]{inputenc}
\usepackage{color}
\usepackage{amssymb,amsmath,amsthm,amscd,amsfonts}
\usepackage{enumerate}
\usepackage[shortlabels]{enumitem}
\usepackage[all]{xy}
\usepackage{hyperref,cleveref,graphics,mathrsfs}

\numberwithin{equation}{section}
\def\hangbox to #1 #2{\vskip3pt\hangindent #1\noindent \hbox to #1{#2}$\!\!$}

\usepackage{hyperref,cleveref,amssymb,mathtools}

\theoremstyle{plain}
\newtheorem{theorem}{Theorem}[section]
\newtheorem{proposition}[theorem]{Proposition}
\newtheorem{corollary}[theorem]{Corollary}
\newtheorem{lemma}[theorem]{Lemma}

\theoremstyle{definition}
\newtheorem{remark}[theorem]{Remark}
\newtheorem{question}[theorem]{Question}

\newtheorem{example}[theorem]{Example}
\newtheorem{definition}[theorem]{Definition}

\newcommand{\marg}[1]{\marginpar{\tiny #1}}     

\DeclareSymbolFont{bbold}{U}{bbold}{m}{n}
\DeclareSymbolFontAlphabet{\mathbbold}{bbold}
\def\one{\mathbbold{1}}

\def\C{{\mathbb C}}
\def\N{{\mathbb N}}
\def\R{{\mathbb R}}

\def\F{{\mathbb F}}

\def\sfrac#1#2{\kern.1em\raise.5ex\hbox{$#1$}
        \kern-.1em/\kern-.05em\lower.25ex\hbox{$#2$}}

\def\vp{\varepsilon}
\def\dim{\operatorname{dim}}

\newcommand{\sign}{\text{\rm sign}}

\newcommand{\fw}{\text{\fw}}

\newcommand{\be}{\begin{equation}}
\newcommand{\ee}{\end{equation}}

\newcommand{\is}{\mathcal{S}}

\newcommand{\uu}{\mathfrak{U}}

\newcommand{\spn}{{\mathrm{span} \, }}

\newcommand{\timur}[1]{{\textcolor{blue}{{\bf T.O:} #1}}}

\title[Stable phase retrieval in function  spaces]{Stable phase retrieval in function  spaces}
\author{D.\ Freeman}
\address{Department of Mathematics and Statistics\\
St Louis University\\
St Louis, MO   USA} \email{daniel.freeman@slu.edu}
\author{T.~Oikhberg}
\address{Dept. of Mathematics, University of Illinois, Urbana IL 61801, USA}
\email{oikhberg@illinois.edu}
\author{B.\ Pineau}
\address{Department of Mathematics\\
University of California at Berkeley
} \email{bpineau@berkeley.edu}
\author{M.A.\ Taylor}
\address{Department of Mathematics\\
University of California at Berkeley
} \email{mitchelltaylor@berkeley.edu}
\date{\today}



\thanks{D.~Freeman  was supported by the Simons Foundation award 706481  and  the NSF award 2154931.  T.~Oikhberg was supported by the NSF award 1912897. \\ }

\begin{document}
\begin{abstract}
    Let $(\Omega,\Sigma,\mu)$ be a measure space, and $1\leq p\leq \infty$. A subspace $E\subseteq L_p(\mu)$ is said to do \emph{stable phase retrieval (SPR)} if there exists a constant $C\geq 1$ such that for any $f,g\in E$ we have 
    \begin{equation}
       \inf_{|\lambda|=1} \|f-\lambda g\|\leq C\||f|-|g|\|.
    \end{equation}
    In this  case, if $|f|$ is known, then $f$ is uniquely determined up to an unavoidable global phase factor $\lambda$; moreover, the phase recovery map is $C$-Lipschitz. Phase retrieval appears in several applied circumstances, ranging from crystallography to quantum mechanics.
    
    In this article, we construct various subspaces doing stable phase retrieval, and make connections with $\Lambda(p)$-set theory. Moreover, we set the foundations for an analysis of stable phase retrieval in general function spaces. This, in particular, allows us to show that H\"older stable phase retrieval implies stable phase retrieval,  improving the stability bounds in a recent article of M.~Christ and the third and fourth authors. We also characterize those compact Hausdorff spaces $K$ such that $C(K)$ contains an infinite dimensional SPR subspace.
\end{abstract}

\maketitle
\allowdisplaybreaks
\tableofcontents

\setlength\parindent{0pt}
\section{Introduction}

There are many situations in mathematics, science, and engineering where the goal is to recover some vector $f$ from $|Tf|$, where $T$ is a linear transformation into a function space.  Note that if $|\lambda|=1$ then it is impossible to distinguish $f$ and $\lambda f$ in this way.  The linear transformation $T$ is said to do  phase retrieval if this ambiguity is the only obstruction to recovering $f$.  That is, 
given a vector space $H$ and function space $X$, a linear operator $T:H\rightarrow X$ does {\em phase retrieval} if whenever $f,g\in H$ satisfy $|T f|=|Tg|$ then $f=\lambda g$ for some scalar $\lambda$ with $|\lambda|=1$.  Phase retrieval naturally arises in situations where one is only able to obtain the magnitude of linear measurements, and not the phase.
Notable examples in physics and engineering which require phase retrieval include X-ray crystallography, electron microscopy, quantum state tomography, and cepstrum analysis in speech recognition. The study of phase retrieval in mathematical physics dates back to at least 1933 when in his seminal work \emph{Die allgemeinen Prinzipien der Wellenmechanik} \cite{pauli1933allgemeinen} W.~Pauli asked whether a wave function is uniquely determined by the probability densities of position and momentum.  In other words, Pauli asked whether $|f|$ and $|\widehat{f}|$ determine $f\in L_2(\mathbb{R})$ up to multiplication by a unimodular scalar. The mathematics of phase retrieval has since grown to be an important and well-studied topic in applied harmonic analysis.\\ 


As any application of phase retrieval would involve error, it is of fundamental importance that the recovery of $f$ from $|Tf|$ not only be possible, but also be stable.  We say that $T$ does {\em stable phase retrieval } if the recovery (up to a unimodular scalar) of $f$ from $|Tf|$ is Lipschitz.  If $X$ is finite dimensional, then $T$ does phase retrieval if and only if it does stable phase retrieval \cite{balan2013,casazza2013}.  However, if $X$ is infinite dimensional and $T$ is the analysis operator of a frame or a continuous frame, then $T$ cannot do stable phase retrieval \cite{MR3656501,MR3554699}.  Here, a collection of vectors $(\psi_t)_{t\in\Omega}\subseteq H$ is a {\em continuous frame} of a Hilbert space $H$ over a measure space $(\Omega,\Sigma, \mu)$ if the map $f\mapsto (\langle f,\psi_t\rangle )_{t\in\Omega}$ is an embedding of $H$ into $L_2(\mu)$.  One of the main goals of this paper is to use the theory of subspaces of Banach lattices to present a unifying framework for stable phase retrieval which encompasses the previously studied cases and allows for stable phase retrieval in infinite dimensions.\\

Let $X=L_p(\mu)$, or, more generally, a Banach lattice. Let $E\subseteq X$ be a  subspace.  We say that $E$ does {\em phase retrieval} as a subspace of $X$ if whenever $|f|=|g|$ for some $f,g\in E$ we have that $f=\lambda g$ for some scalar $\lambda$ with $|\lambda|=1$.  Given a constant $C>0$, we say that  $E$ does {\em $C$-stable phase retrieval} as a subspace of $X$  if
\begin{equation}\label{0.1} 
    \inf_{|\lambda|=1}\|f-\lambda g\|\leq C\big\||f|-|g|\big\|\hspace{1cm}\textrm{ for all } f,g\in E.
\end{equation}
We may define an equivalence relation $\sim$ on $E$ by $f\sim g$ if $f=\lambda g$ for some scalar $\lambda$ with $|\lambda|=1$.  Then, $E$ does  phase retrieval as a subspace of $X$ if and only if the map $f\mapsto |f|$ from $E/\!\sim$ to $X$ is injective.  Furthermore, $E$ does $C$-stable phase retrieval as a subspace of $X$ if and only if the map $f\mapsto |f|$ from $E/\!\sim$ to $X$ is injective and the inverse is $C$-Lipschitz.  By introducing stable phase retrieval into the setting of Banach lattices, we are able to apply established methods from the subject to attack problems in phase retrieval, and conversely we hope that the ideas and questions in phase retrieval will inspire a new avenue of research in the theory of Banach lattices.  Before starting the meat of the paper, we present some additional motivation, give an outline of our major results, and state some of the important ideas and theorems from Banach lattices which we will be applying.  We conclude the paper by listing many open questions concerning stable phase retrieval in this new setting.
\\

\subsection{Motivation and applications}\label{MOT}
The inequality \eqref{0.1} arises in various circumstances. For instance, in crystallography and optics, one seeks to recover an unknown function $F\in L_2(\mathbb{R}^d)$ from the absolute value of its Fourier transform $\widehat{F}$. If one also seeks stability, this translates into an inequality of the form
 \begin{equation}\label{0.2}
       \inf_{|\lambda|=1} \|F-\lambda G\|_{L_2}\leq C\||\widehat{F}|-|\widehat{G}|\|_{L_2},
    \end{equation}
    which one would want to be valid for $F,G$ in a subspace $E\subseteq L_2(\mathbb{R}^d)$ which incorporates the additional constraints $F,G$ are known to satisfy. Using Plancherel's theorem to write $\|F-\lambda G\|_{L_2}=\|\widehat{F}-\lambda \widehat{G}\|_{L_2}$, one sees that the inequality \eqref{0.2} reduces to \eqref{0.1}, up to passing to Fourier space and making the change of notation $f=\widehat{F}$ and $g=\widehat{G}$. We refer the reader to the surveys \cite{MR4094471,MR3674475} and references therein for a further explanation of the importance  of phase retrieval in optics, crystallography, and other areas. In particular, these articles explains why, in practice, physical experiments are often able to measure the magnitude of the Fourier transform, but are unable to measure the phase. 
\\

A second scenario where phase retrieval appears is quantum mechanics. In this case, one wants to identify situations where $|f|$ and $|\widehat{f}|$ determine $f\in L_2(\mathbb{R})$ uniquely. As already mentioned, Pauli asked whether this could true for all $f\in L_2(\mathbb{R})$. However, a counterexample to this conjecture was given in 1944: There exists $f,g\in L_2(\mathbb{R})$ such that $|f|=|g|$ and $|\widehat{f}|=|\widehat{g}|$  but $f$ is not a multiple of $g$. This leads to the natural question of whether one can build ``large" subspaces $G\subseteq L_2(\mathbb{R})$ for which $|f|$ and $|\widehat{f}|$ determine $f\in G\subseteq L_2(\mathbb{R})$ uniquely. By passing to the phase space $L_2(\mathbb{R})\times L_2(\mathbb{R})$, we see that $G$ has the above property if and only if $E:=\{(f,\widehat{f}) : f\in G\}$ does phase retrieval as a subspace of $L_2(\mathbb{R})\times L_2(\mathbb{R})$, i.e., knowing $h,k\in E$ and $|h|=|k|$ implies $h$ is a unimodular multiple of $k$. This also naturally leads to the question of stability of Pauli phase retrieval, by requiring \eqref{0.1} hold on $E$. In this case, using Plancherel's theorem to return to $G$, \eqref{0.1} on $E$ translates into the inequality
\begin{equation}\label{0.3}
     \inf_{|\lambda|=1} \|f-\lambda g\|_{L_2}\leq C\left(\||f|-|g|\|_{L_2}+\||\widehat{f}|-|\widehat{g}|\|_{L_2}\right) \ \text{for}\ f,g\in G.
\end{equation} 
For a non-exhaustive collection of results on Pauli phase retrieval and its generalizations, see \cite{MR3552875,MR4094471,MR1700086,JaRa} and references therein. To our knowledge, the question of stability in the Pauli Problem is essentially unexplored. However, the results presented here in conjunction with \cite{christ2022examples} give a relatively large class of subspaces of $L_2(\mathbb{R}^d)$ satisfying \eqref{0.3}.
\\

Finally, we mention that phase retrieval has grown to become an exciting and important topic of research in frame theory \cite{balan2013,balan2021lipschitz,MR3202304,casazza2017weak,casazza2017norm,DB,MR4094471}. A frame for a separable Hilbert space $H$ is a sequence of vectors $(\phi_j)_{j\in J}$ in $H$ such that there exists uniform bounds $A,B>0$ so that
\begin{equation}\label{E:frame inequality}
    A\|f\|^2\leq \sum_{j\in J}|\langle f,\phi_j\rangle|^2\leq B\|f\|^2\hspace{1cm}\textrm{ for all }f\in H.
\end{equation}
The {\em analysis operator} of a frame $(\phi_j)_{j\in J}$ of $H$ is the map $\Theta:H\rightarrow\ell_2(J)$ given by $\Theta(f)=(\langle f,\phi_j\rangle)_{j\in J}$.  Note that the uniform upper bound $B$ in the frame inequality \eqref{E:frame inequality} guarantees that $\Theta:H\rightarrow \ell_2(J)$ is bounded, and the uniform lower bound $A$ gives that $\Theta$ is an embedding of $H$ into $\ell_2(J)$.  Given a frame $(\phi_j)_{j\in J}$ of $H$, the {\em canonical dual frame} $(\widetilde{\phi}_j)_{j\in J}$ is defined by $\widetilde{\phi}_j=(\Theta^*\Theta)^{-1} \phi_j$ for all $j\in J$ and satisfies
\begin{equation}\label{E:frame recon}
    f= \sum_{j\in J}\langle f, \widetilde{\phi}_j\rangle \phi_j=\sum_{j\in J}\langle f,\phi_j\rangle \widetilde{\phi}_j\hspace{1.5cm}\textrm{ for all }f\in H.
\end{equation}
Frames have many applications and play a fundamental role in signal processing and applied harmonic analysis.  One important reason for this is that the analysis operator $\Theta$ is an embedding of $H$ into $\ell_2(J)$, which allows for the application of filters, thresholding, and other signal processing techniques.  Another reason is that \eqref{E:frame recon} gives a linear, stable, and unconditional reconstruction formula for a vector in terms of the frame coefficients. 
\\

A frame $(\phi_j)_{j\in J}$ is said to do \emph{phase retrieval} if whenever $f,g\in H$ and $(|\langle f,\phi_j\rangle|)_{j\in J}=(|\langle g, \phi_j\rangle |)_{j\in J}$, there exists a unimodular scalar $\lambda$ such that $f=\lambda g.$ A frame is said to do \emph{stable phase retrieval} if there exists a constant $C\geq 1$ such that for all $f,g\in H$,
\begin{equation}\label{0.4}
    \inf_{|\lambda|=1}\|f-\lambda g\|_H\leq C\||\Theta(f)|-|\Theta(g)|\|_{\ell_2(J)}.
\end{equation}
Using the fact that the analysis operator $\Theta:H\to \ell_2(J)$ is an embedding, we see that a frame does stable phase retrieval if and only if the subspace $\Theta(H)\subseteq \ell_2(J)$ does stable phase retrieval in the sense of \eqref{0.1}. In finite dimensions, phase retrieval for frames is automatically stable. However, in infinite dimensions, it is necessarily unstable. As we will see, this is due to the fact that the ambient Hilbert lattice $\ell_2(J)$ is atomic, whereas the construction of SPR subspaces from \cite{calderbank2022stable} is done in the non-atomic lattice $L_2(\mathbb{R})$. For further investigations on the instability of phase retrieval for frames - including generalizations to continuous frames and frames in Banach spaces - see \cite{MR3656501,MR3554699}. 
\\

As mentioned previously, phase retrieval problems  arise in applications when considering an operator $T:H\to X$, which embeds a Hilbert space $H$ into a function space $X$.  In particular, the inequality \eqref{0.2} arises by taking $T$ be the Fourier transform, and \eqref{0.4} arises by taking $T$ to be the analysis operator of a frame. Another  important choice for $T$ is the Gabor transform (see \cite{MR4162323,grohs2021stable} for recent advances in Gabor phase retrieval). As should now be evident, the question of stability for each of these phase retrieval problems can be translated into a special case of \eqref{0.1}, by taking $E:=T(H).$

\subsection{An overview of the results}\label{OVERview}
The  examples from \Cref{MOT} show that the inequality \eqref{0.1} unifies various phase retrieval problems. However, as mentioned previously, phase retrieval for frames is unstable in infinite dimensions, and  it was only recently that the first examples of infinite dimensional SPR subspaces of real $L_2(\mu)$ spaces were constructed \cite{calderbank2022stable}.  The purpose of this article is twofold. First, we construct numerous examples of subspaces of $L_p(\mu)$ doing stable phase retrieval. For this, we use various isometric Banach space techniques, modifications of the  ``almost disjointness" methods in classical Banach lattice theory, random constructions, and analogues of some constructions from harmonic analysis. Secondly, we prove several structural results about SPR subspaces of $L_p(\mu)$, and even general Banach lattices. Notably, both the characterization of real SPR in terms of almost disjoint pairs (\Cref{Thm1}), as well as the equivalence of SPR and its H\"older analogue (\Cref{HSPR==SPR}) hold for all Banach lattices. Our results also extend those in the recent article \cite{christ2022examples}, which uses orthogonality and  combinatorial arguments akin to  Rudin's work \cite{MR0116177} on $\Lambda(p)$-sets to produce examples of subspaces of (real or complex) $L_p(\mu)$ doing H\"older stable phase retrieval. 
\\

We now briefly overview the paper. In \Cref{prelim}, we recall some basic terminology and results from Banach lattice theory in order to make the paper accessible to a wider audience. Most notably, in \Cref{KPD} we collect basic facts related to the Kadec-Pelczynski dichotomy. Such results give structural information about closed subspaces of Banach lattices that are \emph{dispersed}, i.e., that do not contain normalized almost disjoint sequences. As we will show in \Cref{Thm1}, a subspace of a (real) Banach lattice does stable phase retrieval if and only if it does not contain normalized almost disjoint \emph{pairs}. In \Cref{summary}, we collect various facts about dispersed subspaces; finding SPR analogues of these results will occupy much of the paper. In particular, although SPR is much stronger than being  dispersed, in \Cref{Further subspace doing SPR} we will show that every closed infinite dimensional dispersed subspace of an order continuous Banach lattice contains a further closed infinite dimensional subspace doing SPR.  The preliminary section finishes with \Cref{CBL}, which recalls basic facts about complex Banach lattices.
\\

\Cref{GT} collects various results on stable phase retrieval that hold for general Banach lattices. In particular, in \Cref{ADPAS} we make the aforementioned connection between stable phase retrieval and almost disjoint pairs (see \Cref{Thm1}). In \Cref{HSPOV}, we show that if the phase recovery map is H\"older continuous on the ball, then it is Lipschitz continuous on the whole space (\Cref{HSPR==SPR}). This follows from \Cref{thm hspr}, which shows that failure of stable phase retrieval can be witnessed by ``well-separated" vectors. The equivalence between stable phase retrieval and H\"older stable phase retrieval allows us to improve some results from \cite{christ2022examples}, yielding the first examples of infinite dimensional closed subspaces of complex $L_2(\mu)$ doing stable phase retrieval. 
\\

In \Cref{SPR EXAMPLES}, we build infinite dimensional SPR subspaces using a variety of different techniques. In particular, we  prove in \Cref{cor clark} an analogue of statement (iii) of \Cref{summary}; namely,  that for every dispersed subspace $E\subseteq L_p[0,1]$ ($1\leq p\leq \infty)$, we can build a closed subspace $E'\subseteq L_p[0,1]$ isomorphic to $E$, and doing stable phase retrieval. Moreover, for $p<2$ and  $q\in (p,2]$, we will show that any closed subspace of $L_p(\mu)$  isometric to $L_q(\mu)$  does SPR in $L_p(\mu)$, see \Cref{rem-clark}. Regarding sequence spaces, in \Cref{ss:SPR embeddings} we show that $\ell_\infty$ embeds into itself in an SPR way, while no infinite dimensional subspace of $\ell_p$ does SPR when $1\leq p<\infty.$ \Cref{s:intro}  constructs SPR subspaces of r.i.~spaces using random variables. This, in particular, tells us that  subspaces spanned by iid Gaussian and $q$-stable random variables will do SPR in a variety of spaces, including all $L_p$-spaces in which they can be found.  Finally, \Cref{Stable SPR} provides some basic stability properties of SPR subspaces.
\\

\Cref{SPR in LP} contains a study of the structure of SPR subspaces of $L_p(\mu)$, for a finite measure $\mu$. We begin with the aforementioned \Cref{Further subspace doing SPR}, which is  applicable for general order continuous Banach lattices, but for which much of the proof occurs in $L_1(\mu)$. Indeed, the generalization to order continuous Banach lattices follows from the result in $L_1(\mu)$ by arguing via the Kadec-Pelczynski dichotomy.
\\

Note that from the classical results in \Cref{summary} (a)-(d) it follows that if $E$ is dispersed in $L_p(\mu)$ and $1\leq q<p<\infty$, then $E$ may be viewed as a closed subspace of $L_q(\mu)$, and it is dispersed in $L_q(\mu)$. In \Cref{p>2 not down} we show that if $2\leq p<\infty$, there are closed subspaces $E\subseteq L_p(\mu)$ which do SPR (and hence are dispersed in $L_q(\mu)$ for all $1\leq q\leq p$), but fail to do SPR when viewed as a closed subspace of $L_q(\mu)$ for all $1\leq q<p.$ However, by \Cref{p is not always 2}, if $p<2$, then any SPR subspace $E\subseteq L_p(\mu)$ also does SPR when viewed as a closed subspace of $L_q(\mu)$ for any $1\leq q\leq p$. Whether there is an SPR analogue of statement (v) of \Cref{summary} remains an open problem.
\\

\Cref{s:SPR CK} is devoted to the study of infinite dimensional SPR subspaces of $C(K)$. The main result is \Cref{t:C(K) SPR} which states that for a compact Hausdorff space $K$, the space $C(K)$ of continuous functions over $K$ admits a (closed) infinite dimensional  SPR subspace if and only if the Cantor-Bendixson derivative $K'$ of $K$ is infinite. The paper finishes with \Cref{OP}, which discusses various avenues for further research. 

\section{Preliminaries}\label{prelim}
As many of our results hold in the generality of Banach lattices, we briefly summarize some of the standard notations and conventions from this theory. For the most part, our conventions align with the references \cite{AB,MR540367}. Moreover, the statements of our results require minimal knowledge of Banach lattices to understand; it is simply the proofs that use the technology and terminology from this theory.  Unless otherwise mentioned,  all $L_p$-spaces, $C(K)$-spaces and Banach lattices are  real. When a result is applicable for complex scalars, we will explicitly state this. The word ``subspace" is to be interpreted in the vector space sense. If a result requires the subspace to be closed or (in)finite dimensional, we will state this.
\\

Recall that a \emph{vector lattice} is a vector space, equipped with a compatible lattice-ordering (see \cite{AB} for a precise definition). For a vector lattice $X$, the positive cone of $X$ is denoted by $X_+:=\{f\in X : f\geq 0\}.$ The infimum of $f,g\in X$ is denoted by  $f\wedge g$, and the supremum is denoted by $f\vee g$. The modulus of $f$ is defined as $|f|:=f\vee (-f)$, and elements $f,g\in X$ are  said to be \emph{disjoint} if $|f|\wedge |g|=0$. A  \emph{weak unit} is an element $e\in X_+$ for which $|f|\wedge e=0$ implies $f=0$. For a net $(f_\alpha)$ in $X$, the notation $f_\alpha\downarrow 0$ means that $f_\alpha$ is decreasing and has infimum $0$.  A subspace $E\subseteq X$ is a \emph{sublattice} if it is closed under finite lattice operations; it is an \emph{ideal} if $f\in E$ and $|g|\leq |f|$ implies $g\in E$.
\\

A \emph{Banach lattice} is a Banach space which is also a vector lattice, and for which one has the compatibility condition $\|f\|\leq \|g\|$ whenever $|f|\leq |g|.$ Note that the SPR inequality \eqref{0.1} remains well-defined when $L_p(\mu)$ is replaced by an arbitrary Banach lattice. As we will see, several  of our results on SPR are also valid in this level of generality. Common examples of Banach lattices include $L_p$-spaces, $C(K)$-spaces, Orlicz spaces, and various sequence spaces. In this case, the ordering is pointwise, i.e.,  $f\leq g$  means $f(t)\leq g(t)$ for all (or almost all in the case of measurable functions) $t$ in the domain of $f$ and $g$. 
\\

A Banach lattice $X$ is \emph{order continuous} if for each net $(f_\alpha)$ satisfying $f_\alpha\downarrow 0$ we have $f_\alpha\xrightarrow{\|\cdot\|_X}0.$ $L_p$-spaces are order continuous for $1\leq p<\infty$, but $C(K)$-spaces are not (unless they are finite dimensional). To transfer results from $L_1(\mu)$ to more general Banach lattices, we will make use of the \emph{AL-representation} procedure. For this, let $X$ be an order continuous Banach lattice with a weak unit $e$. It is known that $X$ can be represented as an order and norm dense ideal in $L_1(\mu)$ for some finite measure $\mu$. That is, there is a vector lattice isomorphism $T:X\to L_1(\mu)$ such that Range$T$ is an order and norm dense ideal in $L_1(\mu)$. Note that $T$ need not be a norm isomorphism, though $T$ may be chosen to be continuous with $Te=\one$. Moreover, Range$T$ contains $L_\infty(\mu)$ as a norm and order dense ideal. It is common to identify $X$ with Range$T$ and  view $X$ as an ideal of $L_1(\mu)$. Such an inclusion of $X$ into $L_1(\mu)$ is called an \emph{AL-representation} of $X$. We refer to \cite[Theorem 1.b.14]{MR540367} or \cite[Section 4]{MR3666441} for details on AL-representations.

\subsection{The Kadec-Pelczynski dichotomy}\label{KPD}
Here, we briefly recap the literature on subspaces which do not contain almost disjoint normalized sequences. Recall that a sequence $(x_n)$ in a Banach lattice $X$ is said to be a \emph{normalized almost disjoint sequence} if $\|x_n\|_X=1$ for all $n$, and there exists a disjoint sequence $(d_n)$ in $X$ such that $\|x_n-d_n\|_X\to 0$. Following \cite{MR4301527,GMM,Disjointquant}, a closed subspace of a Banach lattice that fails to contain  normalized almost disjoint sequences will be called  \emph{dispersed}.   The classical \emph{Kadec-Pelczynski dichotomy} (c.f.~\cite[Proposition 1.c.8]{MR540367})  states that for a subspace $E$ of an order continuous Banach lattice $X$ with weak unit,  either
\begin{enumerate}
    \item $E$ fails to be dispersed, i.e., $E$ contains an almost disjoint normalized sequence,  or,
    \item $E$ is isomorphic to a closed subspace of $L_1(\Omega,\Sigma,\mu).$
\end{enumerate}
As we will see in \Cref{Thm1}, for real scalars, a subspace does stable phase retrieval if and only if it does not contain normalized almost disjoint \emph{pairs}. Hence, the Kadec-Pelczynski dichotomy will provide a tool to analyze such subspaces.
\\

In $L_p(\mu)$ for $1\leq p<\infty$ and a probability measure $\mu$, the Kadec-Pelczynski dichotomy can be improved. Indeed, we summarize the literature in the following theorem. 
\begin{theorem}\label{summary}
Let $1\leq p<\infty$ and $\mu$ be a probability measure. For a closed subspace $E$ of $L_p(\mu)$, the following are equivalent:
\begin{enumerate}[(a)]
\item $E$ is dispersed, i.e., $E$ contains no almost disjoint normalized sequences;
    \item There exists $0<q<p$ such that $\|\cdot\|_{L_p}\sim \|\cdot\|_{L_q}$ on $E$;
    \item For all $0<q<p$, $\|\cdot\|_{L_p}\sim \|\cdot\|_{L_q}$ on $E$;
    \item $E$ is strongly embedded in $L_p(\mu)$, i.e.,  convergence in measure coincides with norm convergence on $E$.
\end{enumerate}
Moreover,
\begin{enumerate}
    \item For $p\neq 2$, a closed subspace of $L_p[0,1]$ is dispersed if and only if it contains no subspace  isomorphic to $\ell_p$. 
    \item For $p>2$, a closed subspace of $L_p[0,1]$ is dispersed if and only if it is isomorphic to a Hilbert space.
    \item For $p<2$ and any $q\in (p,2]$, there is a closed subspace of $L_p[0,1]$ which is both dispersed and isometric to $L_q[0,1]$.
    \item For $p\neq 2$, $L_p[0,1]$ cannot be written as the direct sum of two dispersed subspaces.
    \item There exists an orthogonal decomposition $L_2[0,1]=E\oplus E^\perp$ with both $E$ and $E^\perp$ dispersed in $L_2[0,1]$. 
\end{enumerate}
\end{theorem}
\begin{proof}
The equivalence of (b), (c) and (d) is \cite[Proposition 6.4.5]{AK}. Other than the isometric portion of statement (iii), the rest of the statements are neatly summarized in \cite[Propositions 3.4 and 3.5]{GMM}, with references to various textbooks for proofs. An isometric embedding of $L_q[0,1]$ into $L_p[0,1]$ for $q\in (p,2)$ is given in \cite[Corollary 2.f.5]{MR540367}. An isometric embedding of $\ell_2$ into $L_p[0,1]$ for $1\leq p<\infty$ is given in \cite[Proposition 6.4.12]{AK}.
\end{proof}
\begin{remark}\label{Lambda p recap}
One of the goals of this article is to find SPR analogues of the results in \Cref{summary}. However, we should mention that the connection between \Cref{summary} and SPR has already been implicitly made in \cite{christ2022examples}. Recall that a subset $\Lambda \subseteq \mathbb{Z}$ is called a \emph{$\Lambda(p)$-set} if the closed subspace generated by the set of exponentials $\{e^{2\pi inx}: n\in \Lambda \}\subseteq L_p(\mathbb{T})$ satisfies the equivalent conditions in \Cref{summary}. Such sets have been deeply studied  \cite{Lambdap,MR1863693,Lphi}, and have many interesting properties. For example, Rudin \cite{MR0116177} showed that for all integers $n>1$,  there are $\Lambda(2n)$-sets that are not $\Lambda(q)$-sets for every $q>2n$. Moreover,  Bourgain \cite{MR989397} extended Rudin's theorem to all $p> 2$.  On the other hand, when $p<2$, and $\Lambda$ is $\Lambda(p)$, then it is automatically $\Lambda(p+\varepsilon)$ for some $\varepsilon>0$ (\cite{MR383005,MR964865}). 
Since $|e^{2\pi inx}|\equiv 1$, complex exponentials cannot do stable phase retrieval. However, by replacing $e^{2\pi inx}$ by  $\sin(2\pi n x)$ or other trigonometric polynomials with non-constant moduli,  \cite{christ2022examples} is able to use combinatorial arguments in the spirit of Rudin to produce SPR subpaces of $L_p(\mu)$ when the dilation set $\Lambda$ is sufficiently sparse. 
\end{remark}

\subsection{Complex Banach lattices}\label{CBL}
Complex Banach lattices are defined as complexifications of real Banach lattices, and in the case of complex function spaces like $C(K)$ and $L_p(\mu)$, agree with the standard definition. More precisely, by a \emph{complex Banach lattice} we mean the complexification $X_\mathbb{C}=X\oplus iX$ of a real Banach lattice, $X$, endowed with the norm $\|x+iy\|_{X_\mathbb{C}}=\| |x+iy|\|_X$, where the modulus $|\cdot|: X_\mathbb{C}\to X_+$ is the mapping given by
\begin{equation}
    |x+iy|=\sup_{\theta\in [0,2\pi]}\{x\cos\theta+y\sin\theta\}, \ \text{for} \ x+iy\in X_\mathbb{C}.
\end{equation}
We refer to \cite[Section 3.2]{MR1921782} and \cite[Section 2.11]{MR0423039} for a proof that the modulus function is well-defined, and behaves as expected.
\\

With the above definition, one can define complex sublattices, complex ideals, etc. However, we will not need this. We do, however, note that if $T:X\to Y$ is a real linear operator between real Banach lattices, then we may define the \emph{complexification} $T_\mathbb{C}:X_\mathbb{C}\to Y_\mathbb{C}$ of $T$ via $T_\mathbb{C}(x+iy)=Tx+iTy.$ The map $T_\mathbb{C}$ is $\mathbb{C}$-linear, bounded, and if $T$ is a lattice homomorphism then $T_{\mathbb{C}}$ preserves moduli, i.e., $T|z|=|T_\mathbb{C}z|$ for $z\in X_\mathbb{C}.$ When we work with complex Banach  lattices $X_\mathbb{C}$, we will use these facts to identify $X_\mathbb{C}$ as a space of measurable functions on some measure space, and then work pointwise. How to do this will be explained later in the paper.
\section{General theory}\label{GT}
In this section, we present several results on (stable) phase retrieval that are valid in general Banach lattices. We begin with the definitions:
\begin{definition}\label{d:PR}
Let $E$ be a subspace of a vector lattice $X$. We say that $E$ does \emph{phase retrieval} if for each $f,g\in E$ with $|f|=|g|$ there is a scalar $\lambda$ such that $f=\lambda g.$
\end{definition}
\begin{definition}\label{d:SPR}
Let $E$ be a subspace of a real or complex Banach lattice $X$. We say that $E$ does \emph{$C$-stable phase retrieval} if for each $f,g\in E$ we have 
\begin{equation}
    \inf_{|\lambda|=1}\|f-\lambda g\|\leq C\||f|-|g|\|.
\end{equation}
If $E$ does $C$-stable phase retrieval for some $C$, we simply say that $E$ does \emph{stable phase retrieval} (\emph{SPR} for short).
\end{definition}
Note that if a subspace $E$ of a real or complex Banach lattice $X$ does $C$-stable phase retrieval, then so does its closure.
\subsection{Connections with almost disjoint pairs and sequences}\label{ADPAS}
When considering whether a subspace $E\subseteq X$ does phase retrieval, there is one obvious obstruction. If $f,g\in E$ are non-zero disjoint vectors, then $|f-g|=|f+g|=|f|+|g|$, but $f-g$ cannot be a multiple of $f+g$.
Hence, if $E$ is to do phase retrieval, then it cannot contain disjoint pairs. Similarly, if $E$ is to do stable phase retrieval, then it cannot contain ``almost" disjoint pairs. As we will now see, in the real case, these are the \emph{only} obstructions to (stable) phase retrieval.

\begin{definition}
Let $E$ be a subspace of a real or complex Banach lattice $X$. We say that $E$ contains \emph{$\varepsilon$-almost disjoint pairs} if there are $f,g\in S_E$ such that $\||f|\wedge |g|\|< \varepsilon.$ If $E$ contains $\varepsilon$-almost disjoint pairs for all $\varepsilon>0$, we say that $E$ contains \emph{almost disjoint pairs}.
\end{definition}

\begin{theorem}\label{Thm1}
Let $E$ be a subspace of a  Banach lattice $X$, $C\geq 1$ and $\varepsilon>0$. Then,
\begin{enumerate}
\item If $E$ does $C$-stable phase retrieval, then it contains no $\frac{1}{C}$-almost disjoint pairs;
\item If $E$  contains no $\varepsilon$-almost disjoint pairs, then it does $\frac{2}{\varepsilon}$-stable phase retrieval.
\end{enumerate}
In particular, $E$ does stable phase retrieval if and only if it does not contain almost disjoint pairs.
\end{theorem}

\begin{proof}
(i)$\Rightarrow$(ii): Suppose that $E$ does $C$-stable phase retrieval, but there are $f,g\in E$ such that $\|f\|=\|g\|=1$ but $\||f|\wedge |g|\|< \frac{1}{C}$. Define $h_1=f+g$ and $h_2=f-g.$ Then since the identity
$$\left||f+g|-|f-g|\right|=2(|f|\wedge|g|)$$
holds in any vector lattice by \cite[Theorem 1.7]{AB}, we have
$$\||h_1|-|h_2|\|=2\||f|\wedge |g|\|< \frac{2}{C}.$$
On the other hand, $h_1+h_2=2f$ has norm $2$, and $h_1-h_2=2g$ also has norm $2$. This contradicts that $E$ does $C$-stable phase retrieval.
\\

(ii)$\Rightarrow$(i): A classical Banach lattice fact (see, e.g., \cite[Remark after Lemma 3.3]{AT}) is that every Banach lattice embeds lattice isometrically into some space of the form
$$\left(\bigoplus_{i\in I}L_1(\Omega_i,\Sigma_i,\mu_i)\right)_\infty.$$
Since  both stable phase retrieval and existence of almost disjoint pairs are invariant under passing to and from closed sublattices, we may assume without loss of generality that $X$ is of this form.
\\

Suppose $E$ does not do $\frac{2}{\varepsilon}$-stable phase retrieval. Find  $f=(f_i),g=(g_i)\in E$ such that $\|f-g\|,\|f+g\|>\frac{2}{\varepsilon}\||f|-|g|\|.$
For each $i\in I$ let 
$$I_i=\{t\in \Omega_i : \sign\left(f_i(t)\right)=\sign\left(g_i(t)\right), \  \text{or one of}\  f_i(t),g_i(t) \ \text{is zero}\}.$$
Then
$$I_i^c:=\Omega_i\setminus I_i=\{t\in \Omega_i : \sign\left(f_i(t)\right)=-\sign\left(g_i(t)\right)\}.$$
We compute that 
$$|f|-|g|=(|f_i|-|g_i|)_{i\in I}=(|{f_i}_{|I_i}|-|{g_i}_{|I_i}|)_{i\in I}+(|{f_i}_{|I_i^c}|-|{g_i}_{|I_i^c}|)_{i\in I}.$$
So, since the modulus is additive on disjoint vectors,
$$\big||f|-|g|\big|=\big(\big||{f_i}_{|I_i}|-|{g_i}_{|I_i}|\big|\big)_{i\in I}+\big(\big||{f_i}_{|I_i^c}|-|{g_i}_{|I_i^c}|\big|\big)_{i\in I}.$$
Now, by definition of $I_i$ we have 
$$\big(\big||{f_i}_{|I_i}|-|{g_i}_{|I_i}\big|\big|)_{i\in I}=\big(\big|{f_i}_{|I_i}-{g_i}_{|I_i}\big|\big)_{i\in I}$$
and 
$$\big(\big||{f_i}_{|I_i^c}|-|{g_i}_{|I_i^c}|\big|\big)_{i\in I}=\big(\big|{f_i}_{|I_i^c}+{g_i}_{|I_i^c}\big|\big)_{i\in I}.$$
Notice next that $d_1:=({f_i}_{|I_i^c}-{g_i}_{|I_i^c})_{i\in I}$ and $d_2:=({f_i}_{|I_i}+{g_i}_{|I_i})_{i\in I}$ are disjoint. Moreover,
\begin{align*}
\|f-g-({f_i}_{|I_i^c}-{g_i}_{|I_i^c})_{i\in I}\| & =\|({f_i}_{|I_i}-{g_i}_{|I_i})_{i\in I}\|=\big\|\big(\big||{f_i}_{|I_i}|-|{g_i}_{|I_i}|\big|\big)_{i\in I}\big\| \\ & \leq \||f|-|g|\|<\frac{\varepsilon}{2}\|f-g\|.
\end{align*}
Similarly,
\begin{align*}
  \|f+g-({f_i}_{|I_i}+{g_i}_{|I_i})_{i\in I}\| & =\|({f_i}_{|I_i^c}+{g_i}_{|I_i^c})_{i\in I}\| =\big\|\big(\big||{f_i}_{|I_i^c}|-|{g_i}_{|I_i^c}|\big|\big)_{i\in I}\big\| \\ & \leq \||f|-|g|\|<\frac{\varepsilon}{2}\|f+g\|.  
\end{align*}

By assumption, we have that both $f+g$ and $f-g$ are non-zero. Hence, by \cite[Lemma 1.4]{AB}, and the fact that $|d_1|\wedge|d_2|=0$ we have 

$$\frac{|f-g|}{\|f-g\|}\wedge \frac{|f+g|}{\|f+g\|}\leq \frac{|f-g-d_1|}{\|f-g\|}\wedge \frac{|f+g|}{\|f+g\|}+\frac{|d_1|}{\|f-g\|}\wedge\frac{|f+g|}{\|f+g\|}$$
$$\leq\frac{|f-g-d_1|}{\|f-g\|}\wedge \frac{|f+g|}{\|f+g\|}+\frac{|d_1|}{\|f-g\|}\wedge\frac{|f+g-d_2|}{\|f+g\|}.$$
It follows that 

$$\|\frac{|f-g|}{\|f-g\|}\wedge \frac{|f+g|}{\|f+g\|}\|\leq \frac{\|f-g-d_1\|}{\|f-g\|}+\frac{\|f+g-d_2\|}{\|f+g\|}<\varepsilon. $$



Thus, we have constructed normalized $\varepsilon$-almost disjoint vectors $\frac{f+g}{\|f+g\|}$ and $\frac{f-g}{\|f-g\|}$ in $E$.
%
\end{proof}

\begin{remark}\label{Rem 1 on SPR}
Implication (i)  of \Cref{Thm1} holds when the Banach lattice $X$ is replaced by any vector lattice equipped with an absolute norm. Here, a norm on a vector lattice $X$  is \emph{absolute} if $\| |f|\|=\|f\|$ for all $f\in X$; see  \cite{MR760617,kreuter2015sobolev,MR1063259} for more information. 
The  proof of \Cref{Thm1} also shows that a subspace of a Banach lattice does phase retrieval if and only if it does not contain disjoint non-zero vectors. A compactness argument then yields that in finite dimensions, phase retrieval implies stable phase retrieval. Indeed, consider the map $S_E\times S_E\to \mathbb{R}$, $(f,g)\mapsto \||f|\wedge |g|\|.$ Then this map is continuous, so its image is compact, which allows one to conclude that the existence of almost disjoint pairs implies the existence of a disjoint pair. In infinite dimensions, it is relatively easy to construct subspaces doing phase retrieval but failing stable phase retrieval.
\end{remark}

\begin{proposition}\label{PR not SPR}
Every infinite dimensional Banach lattice has a closed subspace which does phase retrieval but not stable phase retrieval.
\end{proposition}
\begin{proof}
By \cite[p.~46, Exercise 13]{MR2011364}, any infinite dimensional Banach lattice $X$ contains a normalized disjoint positive sequence, which we shall index as consisting of vectors $(u_i)_{i \in \N}$ and $(v_s)_{s \in \is}$; here, $\is$ denotes the set of all two-element subsets of $\N$ (the order is not important). We fix an injection $\phi : \N^2 \to \N$ and consider the vectors
$$
f_i = u_i + \sum_{j \neq i} 2^{-4\phi(i,j)} v_{\{i,j\}} .
$$
The sum above converges, and we have
$$
\|u_i - f_i\| \leq \vp_i , \, {\textrm{  where  }} \vp_i = \sum_j 2^{-4\phi(i,j)} .
$$
Then $\sum_i \vp_i = \sum_{i,j} 2^{-4\phi(i,j)} \leq \sum_m 2^{-4m} = 1/15$, hence, by \cite[Theorem 1.3.9]{AK}, $(f_i)$ is a Schauder basic sequence. Also, $1 \leq \|f_i\| \leq 16/15$ for each $i$, so this basis is semi-normalized.
We shall show that $E = \overline{\spn}[f_i : i \in \N]$ fails stable phase retrieval, but has phase retrieval.
\\

To show the failure of SPR, let, for $i \neq j$, $\psi(i,j) = \max\{\phi(i,j), \phi(j,i)\}$. Clearly $\psi(i,j) = \psi(j,i)$, and $\lim_j \psi(i,j) = \infty$ for any $i$. Note that  $f_i \wedge f_j =  2^{-4\psi(i,j)} v_{\{i,j\}}$, hence
$$
\| f_i \wedge f_j \| = 2^{-4\psi(i,j)} \underset{i,j \to \infty}{\longrightarrow} 0 .
$$

Next we show that $E$ does phase retrieval. Pick non-zero $f, g \in E$, with $|f| = |g|$; we have to show that $f = \pm g$. To this end, write $f = \sum_i a_i f_i$ and $g = \sum_i b_i f_i$. We can expand
$$
f = \sum_i a_i u_i + \sum_{\{i,j\} \in \is} \big( a_i 2^{-4 \phi(i,j)} + a_j 2^{-4\phi(j,i)} \big) v_{\{i,j\}} ,
$$
and likewise for $g$. Comparing the coefficients with $u_i$, we conclude that, for every $i$, $|a_i| = |b_i|$. By switching signs in front of $f$ and $g$, and by re-indexing, we can assume that $a_1 = b_1 > 0$. We have to show that the equality $a_i = b_i$ holds for every $i > 1$.
\\

The preceding reasoning shows that $a_i = 0$ iff $b_i = 0$. Suppose both $a_i$ and $b_i$ are different from $0$. Comparing the coefficients with $v_{\{i,j\}}$, we see that
$$
\big| 2^{-4\phi(1,i)} a_1 + 2^{-4\phi(i,1)} a_i \big| = \big| 2^{-4\phi(1,i)} b_1 + 2^{-4\phi(i,1)} b_i \big| ,
$$
which is only possible if $\sign \, a_i = \sign \, b_i$.
\end{proof}




\begin{example}
\Cref{Thm1} fails for complex spaces. Indeed, define $E$ as the complex span of $\{(1,1,1),(i,1,-1)\}\subseteq \mathbb{C}^3$, where we equip $\mathbb{C}^3$ with the modulus $|(a,b,c)|:=(|a|,|b|,|c|)$. Clearly, $E$ contains vectors $f,g$ with $|f|=|g|$  but such that $f-\lambda g$ is not zero for any $\lambda\in \mathbb{C}$. Hence, $E$ fails phase retrieval. However, one can easily compute that $E$ contains no disjoint vectors, which by compactness yields the non-existence of almost disjoint vectors. Moreover, as observed in \cite{christ2022examples}, a complex subspace that contains two linearly independent real vectors cannot do complex phase retrieval. In particular, if $E\subseteq X$ is  subspace of a Banach lattice $X$ with $\dim E\geq 2$, then the canonical subspace $E_\mathbb{C}\subseteq X_\mathbb{C}$ fails to do phase retrieval.
\end{example}

\begin{remark}
\Cref{Thm1} shows that for real scalars, the study of subspaces doing stable phase retrieval is  equivalent to the study of subspaces lacking almost disjoint pairs. As mentioned in \Cref{KPD}, there is a vast literature on closed subspaces lacking almost disjoint normalized \emph{sequences}. Clearly, if $E$ contains an almost disjoint normalized sequence, then it fails to do stable phase retrieval. However, the converse is not true. For example, the standard Rademacher sequence $(r_n)$ in $L_p[0,1]$, $1\leq p<\infty$, is dispersed by Khintchine's inequality, but $|r_n|\equiv 1$ for all $n$. Moreover, if one adds a single disjoint vector to a dispersed subspace, one produces a dispersed subspace  failing phase retrieval. Nevertheless, as mentioned in \Cref{OVERview}, many of the results in \Cref{summary} have SPR analogues.
\end{remark}


\subsection{H\"older stable phase retrieval and witnessing failure of SPR on orthogonal vectors}\label{HSPOV}
In \cite{christ2022examples}, the following terminology was introduced in the setting of $L_p$-spaces: A subspace $E$ of a real or complex Banach lattice $X$ is said to do \emph{$\gamma$-H\"older stable phase retrieval with constant $C$} if for all $f,g\in E$ we have
\begin{equation}\label{HSPR IN}
    \inf_{|\lambda|=1}\|f-\lambda g\|_X\leq C\||f|-|g|\|_X^\gamma\left(\|f\|_X+\|g\|_X\right)^{1-\gamma}.
\end{equation}
The utility of this definition arose from a construction in \cite{christ2022examples} of SPR subspaces of $L_4(\mu)$ which are dispersed in $L_6(\mu)$. Applying  certain H\"older inequality arguments, \cite{christ2022examples} was then able to deduce that such subspaces do $\frac{1}{4}$-H\"older stable phase retrieval in $L_2(\mu)$. The idea in \cite{christ2022examples} is to begin with an orthonormal sequence $(r_k)$, and instead of comparing $|f|$ to $|g|$, one compares $|f|^2$ to $|g|^2$.  Assuming the integrability condition $r_k\in L_4(\mu)$ with uniformly bounded norm, and various orthogonality and mean-zero conditions on the products $r_k\overline{r}_j$, the orthogonal expansion $f=\sum_ka_kr_k$  leads to an orthogonal expansion 
$$|f|^2=\sum_{k\neq j}a_k\overline{a}_jr_k\overline{r}_j+\sum_k|a_k|^2 s_k+\|f\|_{L_2}^2\one, \ s_k=|r_k|^2-\one.$$
The products $r_k\overline{r}_j$ encode how the subspace ``sits" in $L_4(\mu)$, i.e., they encode the lattice structure. However, analyzing $|f|^2$ rather than $|f|$ allows one to work algebraically. As was shown in \cite{christ2022examples}, if one imposes appropriate orthogonality conditions, the subspace $E$ spanned by $r_k$ will do stable phase retrieval in $L_4(\mu)$.  \cite{christ2022examples} then gives examples of such $r_k$  built from dilates of a single function $P$, with $|P|$ not identically constant. Verifying that such sequences $(r_k)$ satisfy the required orthogonality conditions is then a combinatorial exercise, using sparseness of the dilates to get non-overlapping supports with respect to the basis expansion. This sparseness naturally leads to $E$ lying in higher $L_p$-spaces, so that by interpolating, one concludes that $E$ does H\"older stable phase retrieval  in $L_2(\mu)$ with $\gamma=\frac{1}{4}$ if $p=6$, and $\gamma\to \frac{1}{2}$ as $p\to \infty$. 
\\

The purpose of this section is to show that - at the cost of dilating the constant -  H\"older stable phase retrieval is equivalent to stable phase retrieval. For real scalars, this can already be deduced from the almost disjoint pair characterization in \Cref{Thm1}. However, the proof below works equally well for complex scalars.  The following theorem was proven in \cite{localphase} for phase retrieval using a continuous frame for a Hilbert space.  We extend it here to subspaces of Banach lattices.
 \begin{theorem}\label{thm hspr}
 Let $X$ be a Banach lattice, real or complex. There exists a universal constant $K=K_X\in [1,\sqrt{2}]$ such that for any linearly independent $f,g\in X$, there exists $f',g'\in \text{span}\{f,g\}$ with 
 \begin{equation}\label{8.4}
     \min_{|\lambda|=1}\|f-\lambda g\|\leq K \min_{|\lambda|=1}\|f'-\lambda g'\|,
 \end{equation}
 and 
 \begin{equation}\label{8.6}
     (\|f'\|^2+\|g'\|^2)^\frac{1}{2}\leq K \min_{|\lambda|=1}\|f'-\lambda g'\|,
 \end{equation}
 and
 \begin{equation}\label{8.5}
     \||f|-|g|\|\geq \||f'|-|g'|\|.
 \end{equation}
 \end{theorem}

 \begin{remark}
 Conditions \eqref{8.4} and \eqref{8.5} state that replacing $(f,g)$ by $(f',g')$ tightens the SPR inequality up to the universal factor $K = K_X$. The condition \eqref{8.6} states that $f'$ and $g'$ are ``almost orthogonal''; it also permits us to witness the failure of SPR on $f', g'$ with controlled norm.
 \\
 
 The constant $K = K_X$ appearing in the proof of the theorem is the supremum of the Banach-Mazur distance between a $2$-dimensional subspace of $X$ and $\ell_2^2$. In general, by John's Theorem, $K_X \leq \sqrt2$, but in certain cases a better estimate can be obtained. For instance, if $X=L_2(\mu)$ then $K=1$. 
 \end{remark}
 
 To prove \Cref{thm hspr}, we need to represent elements of $X$ as measurable functions.
 As mentioned in the proof of \Cref{Thm1}, every (real) Banach lattice $X$ embeds lattice isometrically into a space of the form $\left(\bigoplus_{i\in I}L_1(\Omega_i,\Sigma_i,\mu_i)\right)_\infty.$
 Hence, throughout the proof we can assume that elements of $X$ are functions on a measure space. In the complex case, a similar reduction is possible. Indeed, let $X$ be a complex Banach lattice. By the discussion in \Cref{CBL}, we can assume that $X=Z_\mathbb{C}$ is the complexification of some (real) Banach lattice $Z$. We can then let $T:Z\to \left(\bigoplus_{i\in I}L_1(\Omega_i,\Sigma_i,\mu_i)\right)_\infty$ be a lattice isometric embedding. The complexification $T_\mathbb{C}$ maps $X$ into the complexification of $\left(\bigoplus_{i\in I}L_1(\Omega_i,\Sigma_i,\mu_i)\right)_\infty$. The codomain of this map is still $\left(\bigoplus_{i\in I}L_1(\Omega_i,\Sigma_i,\mu_i)\right)_\infty$,  but now interpreted as a Banach lattice over the complex field (cf.~\cite[Exercises 3 and 5 on page 110]{MR1921782}). Since $T$ is one-to-one, the definition of $T_\mathbb{C}$ tells us that $T_\mathbb{C}$ is one-to-one. Moreover, as mentioned in \Cref{CBL}, $T_\mathbb{C}$ preserves moduli. Finally, by \cite[Lemma 3.18 or Corollary 3.23]{MR1921782},
$T_\mathbb{C}$ preserves norm. Thus, everything in the SPR inequality is preserved, so, analogously to the real case,  we may assume throughout the proof that the complex Banach lattice $X$ is a space of complex-valued functions.

 \begin{proof}[Proof of \Cref{thm hspr}]
Let $Y=\text{span} \{f,g\}$. This is a $2$-dimensional Banach space. Hence, there exists an equivalent norm $\|\cdot\|_H$ such that $(Y,\|\cdot\|_H)$ is Hilbert, and 
$$\|\cdot\|\leq \|\cdot\|_H\leq \sqrt{2}\|\cdot\|.$$
By replacing $g$ by a unimodular scalar times $g$, we assume
$$\min_{|\lambda|=1}\|f-\lambda g\|_H=\|f- g\|_H.$$
 This latter condition is equivalent to $\langle f,g\rangle\geq 0.$ Indeed,
$$\|f-\lambda g\|_H^2=\langle f,f\rangle+\langle g,g\rangle-2\Re\left(\lambda \langle f,g\rangle\right).$$
This is minimized when $\lambda $ is the conjugate phase of $\langle f,g\rangle$. This is minimized when $\lambda=1$ iff $\langle f,g\rangle\geq 0$.
\\

Consider $f_r:=f-r(f+g)$ and $g_r:=g-r(f+g)$ for $r\in [0,1/2]$. We let $R$ be the first instance of $\langle f-r(f+g),g-r(f+g)\rangle=0.$ This is possible since when $r=0$, the inner product is non-negative, and when $r=\frac{1}{2}$, it is negative. Note that 
$$\|f_r-g_r\|_H=\|f-g\|_H.$$
Thus, since $f_R$ and $g_R$ are orthogonal,
$$\min_{|\lambda|=1}\|f_R-\lambda g_R\|_H=\min_{|\lambda|=1}\|f-\lambda g\|_H.$$
We will take $f'=f_R$ and $g'=g_R$. To see \eqref{8.4}, we compute
\begin{equation}
    \begin{split}
      \min_{|\lambda|=1}\|f-\lambda g \|\leq \min_{|\lambda|=1}\|f-\lambda g\|_H=\min_{|\lambda|=1}\|f'-\lambda g'\|_H\leq \sqrt{2}\min_{|\lambda|=1}\|f'-\lambda g'\|.
    \end{split}
\end{equation}
Moreover, as $f'$ and $g'$ are orthogonal in $H$,
\begin{equation}
\sqrt{2}\min_{|\lambda|=1}\|f'-\lambda g'\|\geq \min_{|\lambda|=1}\|f'-\lambda g'\|_H=(\|f'\|_H^2+\|g'\|_H^2)^\frac{1}{2}\geq(\|f'\|^2+\|g'\|^2)^\frac{1}{2}.
\end{equation}
This gives \eqref{8.6}. Note that in the worst case scenario, we have $K\leq \sqrt{2}.$ However, if the Banach-Mazur distance to $\ell_2^2$ is less than $\sqrt{2}$, the constant improves.
\\

We now verify \eqref{8.5}. To see this, we prove \begin{equation}\left||f_r|-|g_r|\right|\leq \left||f|-|g|\right| \ \text{for} \  r\in [0,\frac{1}{2}].
\end{equation}
We represent $X\subseteq L^0(\Omega)$ and let $t\in \Omega$. 
 We will prove that
\begin{equation}\label{cn reduction}\left||f_r(t)|-|g_r(t)|\right|\leq \left||f(t)|-|g(t)|\right| \ \text{for} \  r\in [0,\frac{1}{2}].
\end{equation}
Note that \eqref{cn reduction} is simply a claim that an elementary inequality holds for complex numbers.  Write $f(t)=a+ib$ and $g(t)=c+id$. Multiplying $f(t)$ and $g(t)$ by a unimodular scalar, we rotate so that $d=-b.$ WLOG, $|a|\geq |c|$; then, multiplying by $-1$ if necessary, we also assume $a\geq 0$. We have 
$$f_r(t)=a-r(a+c)+ib, \ \ \ g_r(t)=c-r(a+c)-ib.$$
Now, we note that our assumptions give $\left||f_r(t)|-|g_r(t)|\right|=|f_r(t)|-|g_r(t)|$ for $0\leq r\leq \frac{1}{2}.$ Indeed, $\Im (f_r(t))=-\Im(g_r(t))$ and $\left(\Re(f_r(t))\right)^2\geq \left(\Re(g_r(t))\right)^2 $ for $0\leq r\leq \frac{1}{2}$ by elementary computations. Taking $r=0$, $\left||f(t)|-|g(t)|\right|=|f(t)|-|g(t)|.$ Hence, we must prove
$$|f_r(t)|-|g_r(t)|\leq |f(t)|-|g(t)| \ \text{for}\ r\in [0,\frac{1}{2}].$$
This inequality is true for all $r\geq 0$. Indeed, recall first that $a\geq c$. By the Fundamental Theorem of Calculus, for any convex function $\phi$ and $w \geq 0$, we have $\phi(a-w) - \phi(c-w) \leq \phi(a) - \phi(c)$.
In our case, the function $h(s)=\sqrt{s^2+b^2}$ is convex and $r(a+c)\geq 0$; therefore,
$$|f_r(t)|-|g_r(t)|=h(a-r(a+c))-h(c-r(a+c))\leq h(a)-h(c)=|f(t)|-|g(t)|. \qedhere$$
 \end{proof}
\begin{remark}\label{Rem on SPR Hilbert}
For real or complex $L_2(\mu)$, the  proof of \Cref{thm hspr} shows that $K=1$, and $f'$ is orthogonal to $g'$. Actually, the proof gives a more  local result: If $E$ is a closed subspace of a Banach lattice $X$, and $E$ is $C$-isomorphic to a Hilbert space $H$, then for $f,g\in E$ one can take $K=C$ and $f'$  orthogonal to $g'$ in $H$. 
\end{remark}

\begin{corollary}\label{HSPR==SPR}
Let $E$ be a subspace of a real or complex Banach lattice $X$, and $\gamma\in (0,1]$. If $E$ does $\gamma$-H\"older stable phase retrieval in $X$ with constant $C>0$ then $E$ does stable phase retrieval in $X$ with constant $\sqrt{2}(\sqrt{8}C)^\frac{1}{\gamma}$.
\end{corollary}
\begin{proof}
Let $f,g\in E$ with  $\|f\|=1$ and $\|g\|\leq 1$ such that 
\begin{equation}\label{E:bound}
(\|f\|^2+\|g\|^2)^\frac{1}{2}\leq \sqrt{2} \inf_{|\lambda|=1}\|f-\lambda g\|.
\end{equation}
In particular, 
$$2^{-1/2}\leq \inf_{|\lambda|=1}\|f-\lambda g\|\leq 2.$$  As $E$ does $C$-stable $\gamma$-H\"older phase retrieval, we have that
\begin{equation}
   2^{-1/2}\leq \inf_{|\lambda|=1}\|f-\lambda g\|\leq C\||f|-|g|\|^\gamma (\|f\|+\|g\|)^{1-\gamma}\leq 2^{1-\gamma}C\||f|-|g|\|^\gamma.
\end{equation}
Thus, we have that $C^{1/\gamma}2^{3/(2\gamma)-1}\||f|-|g|\|\geq1 $ and $\inf_{|\lambda|=1}\|f-\lambda g\|\leq 2$. It follows that
\begin{equation}\label{E:sprbound}
\inf_{|\lambda|=1}\|f-\lambda g\|
\leq (2^{3/2} C)^{1/\gamma}\||f|-|g|\|.
\end{equation}
To prove \eqref{E:sprbound} we have assumed that $\|f\|=1$ and $\|g\|\leq 1$.  However, by scaling we have that any $f,g\in E$  which satisfy \eqref{E:bound} also satisfy \eqref{E:sprbound}.
\\

We now consider any pair of linearly independent vectors $x,y\in E$.  By \Cref{thm hspr} there exists $f,g\in E$ which satisfy \eqref{E:bound} such that 
 $$ \min_{|\lambda|=1}\|x-\lambda y\|\leq \sqrt{2} \min_{|\lambda|=1}\|f-\lambda g\|\textrm{ and }     \||x|-|y|\|\geq \||f|-|g|\|.
$$
Thus, we have that
$$\min_{|\lambda|=1}\|x-\lambda y\|\leq 2^{1/2}(2^{3/2} C)^{1/\gamma}\||x|-|y|\|.
$$
This proves that $E$ does $2^{1/2}(2^{3/2} C)^{1/\gamma}$-stable phase retrieval. 
\end{proof}

\begin{remark}
The constant $\sqrt{2}(\sqrt{8}C)^\frac{1}{\gamma}$ in \Cref{HSPR==SPR} arises by using the worst case scenario $K=\sqrt{2}$ from \Cref{thm hspr}. This constant can certainly be optimized; for example, if one also takes into account the distance from $E$ to a Hilbert space.
\end{remark}
To conclude this section we give a simple proof that in finite dimensions, phase retrieval is automatically stable.
\begin{corollary}\label{Cor HSPR1}
Let $X$ be a real or complex Banach lattice, and $E$ a finite dimensional subspace of $X$. If $E$ does phase retrieval, then $E$ does stable phase retrieval.
\end{corollary}
\begin{proof}
The real case has already been dealt with in \Cref{Rem 1 on SPR}, but the argument we provide below works for both real and complex scalars. Indeed, by \Cref{thm hspr}, if $E$ fails to do stable phase retrieval then we can find, for each $N\in \mathbb{N}$, functions $f_N, g_N$ with $\|f_N\|=1$, $\|g_N\|\leq 1$,  
\begin{equation}\label{SPR IN FINITE 2}
     (\|f_N\|^2+\|g_N\|^2)^\frac{1}{2}\leq \sqrt{2} \min_{|\lambda|=1}\|f_N-\lambda g_N\|,
\end{equation}
and 
\begin{equation}\label{STABLE IN FINITE}
2\geq  \min_{|\lambda|=1}\|f_N-\lambda g_N\|>N\||f_N|-|g_N|\|.
\end{equation}

By compactness, after passing to subsequences, we may assume that $f_N\xrightarrow{\|\cdot\|}f$ and $g_N\xrightarrow{\|\cdot\|}g$, for some $f,g\in E$. Since $\|f_N\|=1$ for all $N$, it follows that $\|f\|=1$. Moreover, from \eqref{STABLE IN FINITE} and continuity of lattice operations, we see that $\||f|-|g|\|=0$.  Hence, $|f|=|g|\neq 0$.
Fix a phase $\lambda$. By \eqref{SPR IN FINITE 2}, we have 
\begin{equation*}
     (\|f_N\|^2+\|g_N\|^2)^\frac{1}{2}\leq \sqrt{2} \|f_N-\lambda g_N\|.
\end{equation*}
Passing to the limit, we see that 
\begin{equation*}
    1\leq (\|f\|^2+\|g\|^2)^\frac{1}{2}\leq \sqrt{2}\|f-\lambda g\|.
\end{equation*}
Hence, $f\neq \lambda g$. It follows that $E$ fails to do phase retrieval.
\end{proof}
\begin{remark}
Note that the Banach lattice $X$ in \Cref{Cor HSPR1} is not assumed to be finite dimensional. This is of some note, as, unlike for closed spans, the closed sublattice generated by a finite set can be infinite dimensional.
\end{remark}

\section{Examples}\label{SPR EXAMPLES}

\subsection{Building SPR subspaces via isometric theory}\label{Isometric theory}
As mentioned in \Cref{summary}, when $1\leq p<2$ and $q\in (p,2]$, one can find isometric copies of $L_q[0,1]$ in $L_p[0,1].$ As we will now see, such subspaces must do SPR.

\begin{proposition}\label{rem-clark}
Suppose $p, q \in [1,\infty)$, and either $(1)$ $1 \leq p < q \leq 2$, or $(2)$ $q = 2 < p < \infty$. 
There exists an $\varepsilon>0$ such that if $E\subseteq L_p[0,1]$ is $(1+\varepsilon)$-isomorphic to $F \subseteq L_q[0,1]$, then $E$ does SPR in $L_p[0,1]$.
\end{proposition}
\begin{proof}
We only handle case (1), as (2) is very similar. Suppose, for the sake of contradiction, that $E$ fails SPR. Then by \Cref{Thm1}, $E$ contains $c$-isomorphic copies of $\ell_p^2$, for any $c>1$.
Consequently, for any such $c$ we can find norm one $f, g \in E$ so that $\|f+g\|_{L_p}, \|f-g\|_{L_p} \geq c^{-1} 2^{1/p}$. However, by the Clarkson inequality in $L_q$,
$$
\|f+g\|_{L_{q}}^{q'} + \|f-g\|_{L_q}^{q'} \leq 2 (\|f\|_{L_q}^q + \|g\|_{L_q}^q)^{q'-1} \leq (1+\varepsilon)^{q'} 2^{q'} ,
$$
where $1/q + 1/q' = 1$. However, the left side is $\geq c^{-q'} 2^{1+q'/p}$, and it is easy to see that $1 + q'/p > q'$. Hence, we get a contradiction if $\varepsilon>0$ is sufficiently small.
\end{proof}

\begin{corollary}\label{NCOR}
If either $1 \leq p < q \leq 2$, or $q = 2 < p < \infty$, then $L_p[0,1]$ contains an SPR subspace isometric to $L_q[0,1]$.
\end{corollary}

\begin{proof}
It is well known (see e.g.~\cite[Section 9]{Handbook-intro}) that, under the above conditions, $L_p[0,1]$ contains an isometric copy of $L_q[0,1]$. By \Cref{rem-clark}, that copy does SPR.
\end{proof}

\subsection{Existence of SPR embeddings into sequence spaces}\label{ss:SPR embeddings}

\begin{proposition}\label{into l infty}
If a Banach space $E$ embeds into $\ell_\infty(\alpha)$ for some cardinal $\alpha$ $($which happens, in particular, when $E$ itself has density character $\alpha)$, then there is an isomorphic SPR  embedding of $E$ inside of $\ell_\infty(\alpha)$. 
\end{proposition}

The fact that any Banach space $E$ of density character $\alpha$ embeds isometrically into $\ell_\infty(\alpha)$ is standard. We recall the construction for the sake of completeness: Let $(x_i)_{i \in I}$ be a dense subset of $E$ of cardinality $\alpha$; for each $i$ find $x_i^* \in S_{E^*}$ so that $x_i^*(x_i) = \|x_i\|$. Then $E \to \ell_\infty(\alpha) : x \mapsto (x_i^*(x))_{i \in I}$ is the desired embedding. Similarly, one can show that if $E$ is a dual space, with a predual of density character $\alpha$, then $E$ embeds isometrically into $\ell_\infty(\alpha)$.
\\

To establish \Cref{into l infty}, it therefore suffices to prove:

\begin{lemma}\label{l infty to l infty}
For any cardinal $\alpha$, there exists an isometric SPR embedding of $\ell_\infty(\alpha)$ into itself.
\end{lemma}

To prove \Cref{l infty to l infty}, we rely on the following.

\begin{lemma}\label{l:functionals}
 Suppose $E$ is a $($real or complex$)$ Banach space, and $x, y \in E$ have norm $1$. Then there exists a norm $1$ functional $f \in E^*$ so that $|f(x)| \wedge |f(y)| \geq 1/5$.
\end{lemma}

\begin{proof}
Suppose first that ${\mathrm{dist}} \, (y , \F x) \leq 2/5$ (here $\F$ is either $\R$ or $\C$). Find $t\in \F$ so that $\|y - tx\| \leq 2/5$. By the triangle inequality, $|t| \geq 3/5$. Find $f \in E^*$ so that $\|f\| = 1 = f(x)$. Then $|f(y)| \geq |t| |f(x)| - \|y-tx\| \geq 1/5$. The case of ${\mathrm{dist}} \, (x , \F y) \leq 2/5$ is handled similarly.
\\

Now suppose ${\mathrm{dist}} \, (x , \F y), {\mathrm{dist}} \, (y , \F x) > 2/5$. By Hahn-Banach Theorem, there exist norm one $g, h \in E^*$ so that $g(x) \geq 2/5$, $g(y) = 0$, $h(y) \geq 2/5$, and $h(x) = 0$. Then $f := (g+h)/\|g+h\|$ has the desired properties. Indeed, $\|g+h\| \leq 2$, hence
$$
|f(x)| \geq \frac12 \big( |g(x)| - |h(x)| \big) \geq \frac15 ,
$$
and likewise, $|f(y)| \geq 1/5$.
\end{proof}

\begin{proof}[Proof of \Cref{l infty to l infty}]
For the sake of brevity, we shall use the notation $E = \ell_\infty(\alpha)$, and $E_* = \ell_1(\alpha)$.
Pick a dense set $(f_i)_{i \in I}$ in $S_{E_*}$, with $|I| = \alpha$. 
Define an isometric embedding $J : E \to \ell_\infty(I) : x \mapsto (f_i(x))_{i \in I}$. We shall show that, for every $x, y \in S_E$ and $\varepsilon>0$, there exists $i$ so that $|f_i(x)| \wedge |f_i(y)| \geq 1/5-\varepsilon$. Once this is done, we will conclude that $\| |Jx| \wedge |Jy| \| \geq 1/5$ for any $x, y \in S_E$, which by \Cref{Thm1} tells us that $J$ is indeed an SPR embedding.
\\

By \Cref{l:functionals}, there exists $f \in S_{E^*}$ so that $|f(x)| \wedge |f(y)| \geq 1/5$. By Goldstine's Theorem, there exists $f' \in S_{E_*}$ so that $|f'(x)| \wedge |f'(y)| \geq 1/5-\varepsilon/2$. Find $i$ so that $\|f' - f_i\| \leq \varepsilon/2$. Then
$$
|f_i(x)| \wedge |f_i(y)| \geq |f'(x)| \wedge |f'(y)| - \|f' - f_i\| \geq \frac{1}{5}-\varepsilon,
$$
which proves our claim.
\end{proof}

\begin{remark}\label{r:emb into C of ball}
We can define the canonical embedding of $E$ into $C(B_{E^*})$ (with $B_{E^*}$ equipped with its weak$^*$ topology) by sending $e \in E$ to the function $e^* \mapsto e^*(e)$. The above reasoning shows that this embedding is SPR. For separable $E$, more can be said -  see  \Cref{p:separble into CK} below. 
\end{remark}



\begin{remark}\label{c0}
If an atomic lattice is order continuous (which $\ell_\infty$ of course is not), then the ``gliding hump'' argument shows the non-existence of infinite dimensional dispersed subspaces. The lattice $c$ is not order continuous, but it has no infinite dimensional dispersed subspaces. This is because $c$ contains $c_0$ as a subspace of finite codimension, hence any infinite dimensional subspace of $c$ has an infinite dimensional intersection with $c_0$. 
\end{remark}

Combining the results from  this and the previous subsection, we see that, often, the collection of dispersed subspaces of a Banach lattice coincides with those that do SPR, up to isomorphism (cf.~\Cref{Q1} below). Indeed, we have the following:
\begin{corollary}\label{cor clark}
For every dispersed subspace $E\subseteq L_p[0,1]$ ($1\leq p\leq \infty)$, there exists a closed subspace $E'\subseteq L_p[0,1]$ isomorphic to $E$, and doing stable phase retrieval. The same result holds with $L_p[0,1]$ replaced by $C[0,1]$, $C(\Delta),$  $c$ or any order continuous atomic Banach lattice.
\end{corollary}
\begin{proof}
By \Cref{summary}, for $1\leq p<\infty$ and $p\neq 2$, a closed subspace of $L_p[0,1]$ is dispersed if and only if it contains no subspace  isomorphic to $\ell_p$.  A result of Rosenthal \cite{MR312222} states that for $1\leq p<2$, a subspace of $L_p[0,1]$ that does not contain $\ell_p$ must be isomorphic to a subspace of $L_r$ for some $r\in (p,2]$. By \Cref{NCOR}, one can build an SPR copy of $L_r$ in $L_p$. 
\\

In the case $2\leq p<\infty,$ \Cref{summary} states that any dispersed subspace of $L_p[0,1]$ must be isomorphic to a Hilbert space. By \Cref{NCOR}, $L_p[0,1]$ contains an SPR copy of $\ell_2$.
To deal with the case $p=\infty$, note that $L_\infty[0,1]$ is isomorphic (as a Banach space) to $\ell_\infty$, and use \Cref{l infty to l infty} together with the fact that $\ell_\infty$ lattice isometrically embeds in $L_\infty[0,1]$.
\\

For order continuous atomic lattices and $c$, there are no infinite dimensional dispersed subspaces by \Cref{c0}. The claim for $C[0,1]$ and $C(\Delta)$ will be proven in \Cref{p:separble into CK} below, when we analyze SPR subspaces of $C(K)$-spaces. As we will see in the proof of \Cref{p:separble into CK}, the fact that every separable Banach space embeds into  $C[0,1]$ and $C(\Delta)$ in an SPR fashion ultimately follows from \Cref{r:emb into C of ball}.
\end{proof}

\subsection{Explicit constructions of SPR subspaces using random variables}\label{s:intro}
In this subsection, we construct SPR subspaces of a rather general class of function spaces by considering the closed span of certain independent random variables.  The use of sub-Gaussian random vectors has been widely successful in building random frames for finite dimensional Hilbert spaces which do stable phase retrieval whose stability bound is independent of the dimension \cite{MR3260258,MR3069958,MR3175089,KS,MR3746047}.  However,  different distributions for random variables will allow for the construction of subspaces which do stable phase retrieval and are not isomorphic to Hilbert spaces.
We begin by presenting a technical criterion for SPR.

\begin{proposition}\label{p:criterion-abstr}
 Suppose $X$ is a Banach lattice of measurable functions on a probability measure space $(\Omega,\mu)$ which contains the indicator functions and has the property that for every $\vp>0$ there exists $\delta = \delta(\vp) > 0$ so that $\|\chi_S\| > \delta$ whenever $\mu(S) > \vp$. Suppose, furthermore, that $E$ is a subspace of $X$, which has the following property: There exist $\alpha > 1/2$ and $\beta > 0$ so that, for any norm one $f\in E$, we have
 \begin{equation}
  \mu \left(\big\{ \omega \in \Omega : |f(\omega)| \geq \beta \big\}\right) \geq \alpha .
 \end{equation}
 Then $E$ is an SPR-subspace.
\end{proposition}

\begin{proof}
 Suppose $f, g \in E$ have norm $1$. By the Inclusion-Exclusion Principle,
 $$
 \mu\left( \big\{ \omega \in \Omega : |f(\omega)| \geq \beta , |g(\omega)| \geq \beta  \big\}\right) \geq 2 \alpha - 1 .
 $$
 Thus, $\| |f| \wedge |g| \| \geq \beta \delta(2\alpha-1)$.
\end{proof}

The above proposition is applicable, for instance, when $X$ is a rearrangement invariant (r.i.~for short; see \cite{LT2} for an in-depth treatment) space on $(0,1)$, equipped with the canonical Lebesgue measure $\lambda$. Examples include $L_p$ spaces, and, more generally, Lorentz and Orlicz spaces (once again, described in great detail in \cite{LT2}; for Lorentz spaces, see also \cite{Dilworth-handbook}). Below we describe some SPR subspaces, spanned by independent identically distributed random variables.  
\\

Suppose $f$ is a random variable, realized as a measurable  function on $(0,1)$ (with the usual Lebesgue measure $\lambda$). Then independent copies of $f$ -- denoted by $f_1, f_2, \ldots$ -- can be realized on $((0,1),\lambda)^{\aleph_0}$. By Caratheodory's Theorem (see e.g.~\cite[p.~121]{Lacey74}), there exists a measure-preserving bijection between $((0,1),\lambda)^{\aleph_0}$ and $((0,1),\lambda)$. Therefore, we view $f_1, f_2, \ldots$ as functions on $(0,1)$.
\\

Suppose now that, in the above setting, the following statements hold:
\begin{enumerate}
    \item $f$ belongs to $X$, and has norm one in that space;
    \item There exists $r$ so that, if $f_1, \ldots, f_n$ are independent copies of $f$, and $\sum_i |a_i|^r = 1$, then $\sum_i a_i f_i$ is equidistributed with $f$;
    \item There exists $\beta > 0$ so that ${\mathbb{P}}(|f| > \beta) > 1/2$.
\end{enumerate}
In this situation, if $f_1, f_2, \ldots$ are independent copies of $f$ (viewed as elements of $X$, per the preceding paragraph), then $\overline{\spn}[f_i : i \in \N]$ is an SPR copy of $\ell_r$ in $X$.
\\

We should mention two examples of random variables with the above properties: Gaussian ((ii) holds with $r=2$) and $q$-stable ($q \in (1,2)$; (ii) holds with $r=q$). The details  can be found in \cite[Section 6.4]{AK}. For the Gaussian variables, the probability density function is $d_f(x) = c e^{-x^2/2}$, with $c$ depending on the normalization. For the $q$-stable variables with characteristic function $t \mapsto c e^{-|t|^q}$ (with $c$ ensuring normalization), the Fourier inversion formula gives the density function
$$
d_f(x) = \frac{c}\pi \int_0^\infty \cos(tx) e^{-t^q} \, dt .
$$
In both cases, $d_f$ is continuous (in the latter case, due to Dominated Convergence Theorem), hence there exists $\beta > 0$ so that 
$$\mathbb{P}(|f| > \beta) = 1 - \int_{-\beta}^\beta d_f > \frac34.$$ 

It is known that Gaussian random variables belong to $L_p$ for $p \in [1,\infty)$, while the $r$-stable random variables ($1 < r < 2$) lie in $L_p$ if and only $p \in [1,r)$. Moreover, the results from \cite[p.~142-143]{LT2} tell us that $L_s(0,1) \subset L_{p,q}(0,1)$ for $s > p$ (this is a continuous inclusion, not an isomorphic embedding). If $r > p$, then the $r$-stable variables belong to $L_{p,q}(0,1)$ (indeed, take $s \in (p,r)$; then the $r$-stable variables live in $L_s(0,1)$, which in turn sits inside of $L_{p,q}(0,1)$). Likewise, one shows that any Lorentz space $L_{p,q}(0,1)$ contains Gaussian random variables.
\\

The above reasoning implies:

\begin{proposition}\label{p:SPR in Lorentz}
Suppose $1 \leq p < \infty$ and $1 \leq q \leq \infty$ $($when $p=1$, assume in addition $q<\infty)$. Then $L_{p,q}(0,1)$ contains a copy of $\ell_2$ that does SPR. If, in addition, $1 \leq p < r < 2$, then $L_{p,q}(0,1)$ contains a copy of $\ell_r$ that does SPR. 
\end{proposition}
\subsection{Stability of SPR subspaces under ultraproducts and small perturbations}\label{Stable SPR}

We show that SPR subspaces are stable under ultraproducts, and under small perturbations (in the sense of Hausdorff distance). These results hold for both real and complex spaces.

\begin{proposition}\label{stable under UF}
Suppose $\uu$ is an ultrafilter on a set $I$, and, for each $i \in I$, $E_i$ is a $C$-SPR subspace of a Banach lattice $X_i$. Then $\prod_\uu E_i$ is a $C$-SPR subspace of $\prod_\uu X_i$.
\end{proposition}

We refer the reader to \cite{Hei} or \cite[Chapter 8]{DJT} for information on ultraproducts of Banach spaces and Banach lattices.

\begin{proof}
We have to show that, for any $x, y \in \prod_\uu E_i$, there exists a modulus one $\lambda$ so that $\|x - \lambda y\| \leq C \| |x| - |y| \|$. To this end, find families $(x_i)$ and $(y_i)$, representing $x$ and $y$ respectively. Then for each $i$ there exists $\lambda_i$ so that $|\lambda_i| = 1$ and $\|x_i - \lambda_i y_i\| \leq C \| |x_i| - |y_i| \|$. As ultraproducts preserve lattice operations, $|x|$ and $|y|$ are represented by $(|x_i|)$ and $(|y_i|)$, respectively, hence $\| |x| - |y| \| = \lim_\uu \| |x_i| - |y_i| \|$. By the compactness of the unit torus, there exists $\lambda = \lim_\uu \lambda_i$, with $|\lambda| = 1$. Then $\|x - \lambda y\| = \lim_\uu \|x_i - \lambda_i y_i\|$, which leads to the desired inequality.
\end{proof}

\begin{remark}
\Cref{stable under UF} can be used to give an alternative proof of \Cref{NCOR}.
First find a family of  finite dimensional subspaces $F_k \subseteq L_q(0,1)$, ordered by inclusion, so that $\cup_k F_k$ is dense in $L_q(0,1)$, and each $F_k$ is isometric to $\ell_q^{n_k}$ for some $n_k$ (one can, for instance, take subspaces spanned by certain step functions). A reasoning similar to that of  \cite[Theorem 8.8]{DJT} permits us to find a free ultrafilter ${\mathfrak{U}}$ so that $\prod_{\mathfrak{U}} F_k$ contains an isometric copy of $L_q(0,1)$. {\it A fortiori}, $\prod_{\mathfrak{U}} \ell_q$ contains an isometric copy of $L_q(0,1)$ (call it $E$). 
\\

\Cref{p:SPR in Lorentz} proves that $L_p(0,1)$ contains a subspace, isometric to $\ell_q$ (spanned by Gaussian random variables for $q=2$, $q$-stable random variables for $q<2$) which does SPR. By \Cref{stable under UF}, $\prod_{\mathfrak{U}} \ell_q$ embeds isometrically into $\prod_{\mathfrak{U}} L_p(0,1)$, in an SPR fashion. By \cite{Hei}, $\prod_{\mathfrak{U}} L_p(0,1)$ can be identified (as a Banach lattice) with $L_p(\Omega,\mu)$, for some measure space $(\Omega,\mu)$. 
\\

Let $X$ be the (separable) sublattice of $L_p(\Omega,\mu)$ generated by $E$. By \cite[Corollary 1.b.4]{MR540367}, $X$ is an $L_p$ space. \cite[Corollary, p.~128]{Lacey74} gives a complete list of all separable $L_p$ spaces; all of them lattice embed into $L_p(0,1)$. Thus, we have established the existence of an SPR embedding of $E = L_q(0,1)$ into $L_p(0,1)$.
\end{remark}

To examine stability of SPR under small perturbations, we introduce the notion of \emph{one-sided Hausdorff distance} between subspaces of a given Banach space. If $E, F$ are subspaces of $X$, define $d_{1H}(E,F)$ as the infimum of all $\delta > 0$ so that, for every $x \in F$ with $\|x\| \leq 1$ there exists $x' \in E$ with $\|x-x'\| < \delta$ (this ``distance'' is not reflexive, hence ``one-sided''). Note also that, for $x$ as above, there exists $x'' \in E$ with $\|x''\| = \|x\|$ and $\|x - x''\| < 2 \delta$; indeed, one can take $x'' = \frac{\|x\|}{\|x'\|} x'$.
\\

By ``symmetrizing'' $d_{1H}$, we obtain the classical \emph{Hausdorff distance}: if $E$ and $F$ are subspaces of $X$, let $d_H(E,F) = \max\{ d_{1H}(E,F), d_{1H}(F,E)\}$. For interesting properties of $d_H$, see \cite{Braga}, and references therein.

\begin{proposition}\label{stability under 1H}
Suppose $E$ is an SPR subspace of a Banach lattice $X$. Then there exists $\delta > 0$ so that any subspace $F$ with $d_{1H}(E,F) < \delta$ is again SPR.
\end{proposition}

From this we immediately obtain:

\begin{corollary}
For any Banach lattice $X$, the set of its SPR subspaces is open in the topology determined by the Hausdorff distance.
\end{corollary}
\begin{remark}
See \cite[Proposition 3.10]{GMM} for a similar stability result for dispersed subspaces of a Banach lattice.
\end{remark}

\begin{proof}[Proof of \Cref{stability under 1H}]
Suppose $E$ does $C$-SPR. We shall show that, if $d_{1H}(E,F) < 1/(2\sqrt2(C+1))$, then $F$ does $C'$-SPR, with
$$
\frac1{C'} = \frac1C \Big( \frac1{\sqrt2} - 2 d_{1H}(E,F) \Big) - 2 d_{1H}(E,F) .
$$

Suppose, for the sake of contradiction, that $F$ fails to do $C'$-SPR. Find $f,g \in F$ so that $\min_{|\lambda|=1} \|f-\lambda g\| = 1$ and $\| |f| - |g| \| = c < 1/C'$. By \Cref{thm hspr}, we can find $f',g' \in F$ so that 
$$ \min_{|\lambda|=1}\|f'- \lambda g'\| \geq \frac1{\sqrt2}, \| |f'| - |g'| \| \leq c, {\textrm{  and   }} \|f'\| + \|g'\| \leq 2.$$

For any $\delta > d_{1H}(E,F)$, there exist $f'', g'' \in E$ so that $\|f'' - f'\| < \delta \|f'\|$ and $\|g'' - g'\| < \delta \|g'\|$. The triangle inequality implies:
\begin{align*}
    & \| |f''| - |g''| \| \leq \| |f'| - |g'| \| + \delta ( \|f'\| + \|g'\| ) \leq c + 2\delta ; \\
    & \min_{|\lambda|=1} \|f'' - \lambda g''\| \geq \min_{|\lambda|=1} \| f' - \lambda g' \| - \delta ( \|f'\| + \|g'\| ) \geq \frac1{\sqrt2} - 2 \delta .
\end{align*}
As $E$ does $C$-SPR, we conclude that 
$$
\frac1{\sqrt2} - 2 \delta \leq C (c + 2\delta) ,
$$
and consequently,
$$
\frac1{\sqrt2} - 2 d_{1H}(E,F) \leq C ( c + 2d_{1H}(E,F) ) < C \Big( \frac1{C'} + 2d_{1H}(E,F) \Big) ,
$$
which contradicts our choice of $C'$.
\end{proof}

R.~Balan proved that frames which do stable phase retrieval for finite dimensional Hilbert spaces are stable under small perturbations \cite{balan2013}. The following extends this to infinite dimensional subspaces of Banach lattices.

\begin{corollary}\label{c:perturb basis}
Suppose $(e_i)$ is a semi-normalized basic sequence in a Banach lattice $X$, so that $\overline{\spn}[e_i: i \in \N]$ does SPR in $X$. Then there exists $\vp > 0$ so that if $(f_i)\subseteq X$ and $\sum_i \|e_i - f_i\| < \vp$ then $\overline{\spn}[f_i : i\in \N]$ does SPR in $X$.
\end{corollary}

\begin{remark}\label{r:approximation}
In real $L_2$, \Cref{c:perturb basis} can be strengthened. Suppose $(e_i)$ is a sequence of normalized independent mean-zero random variables, spanning an SPR-subspace of $L_2$. Then there exists an $\varepsilon > 0$ with the following property: if $(f_i)$ is a collection of  normalized independent mean-zero random variables so that $(e_i, f_j)$ are independent whenever $i \neq j$, and $\sup_i \|e_i - f_i\| \leq \varepsilon$, then $\spn[f_i : i \in \N] \subseteq L_2$ does SPR as well.
For the proof, recall that there exists $\gamma > 0$ so that the inequality $\| |u| \wedge |v| \| \geq \gamma$ holds for any norm one $u, v \in \spn[e_i : i \in \N]$. Let $\varepsilon = \gamma/4$.
 We will  show that, for any norm one $x, y \in F = \spn[f_i : i \in \N]$, we have $\| |x| \wedge |y| \| \geq \gamma/2$.
\\

 Write $x = \sum_i \alpha_i f_i$ and $y = \sum_i \beta_i f_i$, and define $x' = \sum_i \alpha_i e_i$, $y' = \sum_i \beta_i e_i$. Then
 $$
 \|x-x'\|^2 = \big\| \sum_i \alpha_i (f_i - e_i) \big\|^2 = \sum_i |\alpha_i|^2 \|f_i - e_i\|^2 \leq \varepsilon^2 \sum_i \alpha_i^2 = \varepsilon^2 .
 $$
 Similarly, $\|y - y'\| \leq \varepsilon$. Therefore, $\| |x| - |x'| \|, \| |y| - |y'| \| \leq \varepsilon$, hence $\| |x| \wedge |y| \| \geq \| |x'| \wedge |y'| \| - 2 \varepsilon$.
 But $\| |x'| \wedge |y'| \| \geq \gamma$, hence $\| |x| \wedge |y| \| \geq \gamma/2$.
\end{remark}

\section{SPR in $L_p$-spaces}\label{SPR in LP}
In this section, we investigate the relations between dispersed and SPR subspaces of $L_p$, as well as the relation between doing SPR in $L_p$ versus doing SPR in $L_q$.
\begin{theorem}\label{Further subspace doing SPR}
Every  infinite dimensional dispersed subspace of an order continuous Banach lattice $X$  contains a further closed infinite dimensional subspace that does SPR.
\end{theorem}
\begin{proof}
We first prove the claim for $L_1(\Omega,\mu)$, with $\mu$ a finite measure. Let $E$ be a closed infinite dimensional subspace of $L_1(\Omega,\mu)$ containing no  normalized almost  disjoint sequence. By \Cref{summary}, $E$ also does not contain $\ell_1$. By \cite{MR636897}, every closed infinite dimensional subspace of $L_1(\Omega,\mu)$  almost isometrically contains $\ell_r$ for some $1\leq r\leq 2$. Since $E$ does not contain $\ell_1$, it follows that there exists $r>1$ such that for all $\vp>0$, $\ell_r$ is $(1+\vp)$-isomorphic to a subspace of $E$.  Let $\alpha>0$ be such that $\ell_1^2$ is not $(1+\alpha)$-isomorphic to a subspace of $\ell_r$. Such an $\alpha$ exists by the Clarkson argument in \Cref{rem-clark}.
We now claim that for $0<\varepsilon<\alpha$, every subspace of $L_1$ that is $(1+\varepsilon)$-isomorphic to $\ell_r$ must do stable phase retrieval. Indeed, if $E$ failed SPR, it would contain for all $\gamma>0$ a $(1+\gamma)$-copy of $\ell_1^2.$  Thus, for all $\gamma>0$, we have that $\ell_1^2$ is $(1+\gamma)(1+\varepsilon)$-isomorphic to a subspace of $\ell_r$.
However, this gives a contradiction if $\gamma>0$ is small enough such that $(1+\gamma)(1+\varepsilon)<1+\alpha$. 
\\


Now let $E$ be a closed infinite dimensional dispersed subspace of an order continuous Banach lattice $X$. Replacing $E$ be a separable subspace of $E$, we may assume that $E$ is separable. Using that every closed sublattice of an order continuous Banach lattice is order continuous, replacing $X$ by the closed sublattice generated by $E$ in $X$, we may assume that $X$ is separable. It follows in particular that $X$ has a weak unit. By the AL-representation theory, there exists a finite measure space $(\Omega,\mu)$ such that $X$ can be represented as an ideal of $L_1(\Omega,\mu)$ satisfying
\begin{enumerate}
    \item $X$ is dense in $L_1(\Omega,\mu)$ and $L_\infty(\Omega,\mu)$ is dense in $X$;
    \item $\|f\|_1\leq \|f\|_X$ and $\|f\|_X\leq 2\|f\|_\infty$ for all $f\in X$.
\end{enumerate}
 Since $E$ contains no almost disjoint normalized sequence, the Kadec-Pelczynski dichotomy \cite[Proposition 1.c.8]{MR540367} guarantees that $\|\cdot\|_X\sim \|\cdot\|_{L_1}$ on $E$. In particular, we may view $E$ as a closed infinite dimensional subspace of $L_1(\mu)$. We claim that $E$ contains no almost disjoint sequence when viewed as a subspace of $L_1$.  Indeed, suppose there exists a sequence $(x_n)$ in $E$ with $\|x_n\|_{L_1}=1$ for all $n$, and a disjoint sequence $(d_n)$ in $L_1$ with $\|x_n-d_n\|_{L_1}\to 0$. Then in particular, $x_n$ converges to $0$ in measure. By \cite[Theorem 4.6]{MR3688941}, $x_n\xrightarrow{un}0$ in $X$.  That is, for all $u\in X$, we have that $\||x_n|\wedge |u|\|_X\to 0$. Thus, by \cite[Theorem 3.2]{MR3688941} there exists a subsequence $(x_{n_k})$ and a disjoint sequence $(d_k)$ in $X$ such that $\|x_{n_k}-d_k\|_X\to 0$. Since $\|x_n\|_{L_1}=1$ and $\|\cdot\|_X\sim \|\cdot\|_{L_1}$ on $E$, this contradicts that $E$ contains no normalized almost disjoint sequence.
 \\
 
 By the beginning part of the proof, we may select an infinite dimensional closed subspace $E'$ of $E$ that does SPR in $L_1$. In other words, there exists $\varepsilon>0$ such that  for all $f,g\in E'$ with $\|f\|_{L_1}=\|g\|_{L_1}=1$ we have 
 $$\||f|\wedge |g|\|_{L_1}\geq \varepsilon.$$
 Since $\|\cdot\|_X\sim \|\cdot\|_{L_1}$ on $E$, the same is true on $E'$, so we may view $E'$ as a closed infinite dimensional subspace of $X$. We claim that it contains no normalized almost disjoint pairs. Indeed, if $f,g\in E'$ with $\|f\|_X=\|g\|_X=1$, then $\|f\|_{L_1}\sim \|g\|_{L_1}\sim 1$. Now, using that $E'$ does SPR in $L_1$ and property (ii) of the embedding, we have 
 $$\||f|\wedge |g|\|_{X}\geq \||f|\wedge |g|\|_{L_1}\gtrsim \varepsilon.$$
 Thus, $E'$ contains no normalized almost disjoint pairs when viewed as a subspace of $X$. It follows that $E'$ does SPR in $X$.
\end{proof}
\begin{question}
With \Cref{cor clark} and \Cref{Further subspace doing SPR} in mind, we ask the following: If a Banach lattice $X$ contains an infinite dimensional dispersed subspace $E$, does it contains an infinite dimensional SPR subspace? If so, can we construct an infinite dimensional SPR subspace $E'$ with $E'\subseteq E\subseteq X$?
\end{question}

Our next results are motivated by the equivalence between statements (a)-(d) in \Cref{summary} and the discussion in \Cref{Lambda p recap}. Note that it follows from \Cref{summary} (a)-(d) that if $E$ is dispersed in $L_p(\mu)$ and $1\leq q<p$, then $E$ may be viewed as a closed subspace of $L_q(\mu)$, and it is dispersed in $L_q(\mu)$. It is then natural to ask the following question: Let $\mu$ be a finite measure and $1\leq q<p$. Let $E$ be a  subspace of $L_p(\mu)\subseteq L_q(\mu)$. What is the relation between $E$  doing SPR in $L_p(\mu)$ versus $E$ doing SPR in $L_q(\mu)$? It is easy to see that if $E$ does SPR in $L_q(\mu)$, then $E$ does SPR in $L_p(\mu)$ if and only if $\|\cdot\|_{L_p}\sim \|\cdot\|_{L_q}$ on $E$.  We will now show that $E$ doing SPR in $L_p(\mu)$ does not imply $E$ does SPR in $L_q(\mu)$, even though the property of being dispersed passes from $L_p(\mu)$ to $L_q(\mu)$.

\begin{theorem}\label{p>2 not down}
For all $2\leq p<\infty$ there exists a closed subspace $E\subseteq L_p[0,1]$ such that $E$ does stable phase retrieval in $L_p[0,1]$ but  $E$ fails to do stable phase retrieval in $L_q[0,1]$ for all $1\leq q<p$.
\end{theorem}

\begin{proof}
Let $2\leq p<\infty$. It will be convenient to build the subspace $E\subseteq L_p[0,2]$ instead of $L_p[0,1]$.  Let $(r_j)_{j=1}^\infty$ be the Rademacher sequence of independent, mean-zero, $\pm 1$ random variables on $[0,1]$.  For all $j\in\N$, we let $g_j=r_j+2^{j/p}\one_{[1+2^{-j},1+2^{-j+1})}$.
Let $E=\overline{\spn}_{j\in\N} g_j$. 
\\

We first prove for all $1\leq q<p$ that $E$ fails to do stable phase retrieval in $L_q[0,2]$.  We have for all $j\neq i$ that 
$\|g_j-g_i\|_{L_q}^q=\|g_j+g_i\|_{L_q}^q\geq 2^{q-1}$.  On the other hand, $|r_j|=|r_{j+1}|$ and $\lim\|2^{j/p}\one_{[1+2^{-j},1+2^{-j+1})}\|^q_{L_q}=0$.  Thus, $\lim \||g_j|-|g_{j+1}|\|_{L_q}^q=0$.  This shows that $E$ fails to do stable phase retrieval in $L_q[0,2]$.
\\

We now prove that $E$ does stable phase retrieval in $L_p[0,2]$. Note that by Khintchine's Inequality there exists $B\geq 0$ so that $(\sum |a_j|^2)^{1/2}\leq \|\sum a_j r_j\|_{L_p}\leq B(\sum |a_j|^2)^{1/2}$ for all scalars $(a_j)\in\ell_2$.  Thus, we have for all $f=\sum a_j r_j$ and
$x=f+\sum a_j2^{j/p}\one_{[1+2^{-j},1+2^{-j+1})}\in E$ that
\begin{align*}
\|f \|^p_{L_2([0,1])}&\leq\|x\|^p_{L_p([0,2])}=\|\sum a_j r_j\|_{L_p([0,1])}^p+\sum |a_j|^p\\
&\leq B^p (\sum |a_j|^2)^{p/2}+(\sum |a_j|^2)^{p/2}\\
&=(B^p+1)\|f\|^p_{L_2([0,1])}.
\end{align*}

This computation shows that the map $g_j\mapsto r_j$ extends linearly to a map $E\subseteq L_p[0,2]\to L_2[0,1]$, $x\mapsto f$,  establishing an isomorphism between $E$ and a Hilbert space. By Theorem \ref{thm hspr} and \Cref{Rem on SPR Hilbert} it suffices to prove that there exists a constant $\delta>0$ so that if $x,y\in E$ and $f,g\in L_2[0,1]$ with $f=\one_{[0,1]}x$ and $g=\one_{[0,1]}y$ such that $\|f\|_{L_2}=1$, $\|g\|_{L_2}\leq 1$, and $\langle f,g\rangle=0$ then $\||x|-|y|\|_{L_p}\geq \delta$. \\ 

We now claim that it suffices to prove that there exists $\vp>0$ such that 
\begin{equation}
    {\textrm{if   }} \||x\one_{(1,2)}|-|y\one_{(1,2)}|\|_{L_p}<\vp {\textrm{   then   }} \||f|^2-|g|^2\|_{L_2}^2\geq\delta .
    \label{eq:if-then}
\end{equation}
Indeed, as all the $L_q$ norms are equivalent on the span of the Rademacher sequence, there exists a uniform constant $K>0$ so that the following holds:
\begin{align*}
    \||f|^2-|g|^2\|_{L_2}^2&=\int (|f|^2-|g|^2)^2\\
    &=\int (|f|-|g|)(|f|+|g|)(|f|^2-|g|^2)\\
     &\leq \||f|-|g|\|_{L_2} \|(|f|+|g|)(|f|^2-|g|^2)\|_{L_2}\\
     &\leq K\||f|-|g|\|_{L_2} \leq K\||f|-|g|\|_{L_p}.
\end{align*}
Here, the constant $K$ comes from bounding 
\begin{equation}
    \label{eq:bound on product}
\|(|f|+|g|)(|f|^2-|g|^2)\|_{L_2}\leq K.
\end{equation}
To get this upper estimate, note that, by H\"older's Inequality,
\begin{align*}
    &
    \|(|f|+|g|)(|f|^2-|g|^2)\|_{L_2}=\|(|f|+|g|)(|f|+|g|)(|f|-|g|)\|_{L_2}
    \\ &
    \leq \||f|+|g|\|_{L_6}^2 \||f|-|g|\|_{L_6} ,
\end{align*}
hence, by Triangle Inequality,
\begin{equation}
    \label{eq:after triangle ineq}
\|(|f|+|g|)(|f|^2-|g|^2)\|_{L_2} \leq \big( \|f\|_{L_6} + \|g\|_{L_6} \big)^3 .
\end{equation}
 Further, both $f$ and $g$ belong to the span of independent Rademachers, on which all the $L_p$ norms are equivalent (for finite $p$). Since we know that $\|f\|_{L_2} = 1$ and $\|g\|_{L_2} \leq 1$, this gives a bound for the right-hand side of \eqref{eq:after triangle ineq}, which, in turn, implies \eqref{eq:bound on product}.
\\

To finish the proof of the claim, note that if $\||x\one_{(1,2)}|-|y\one_{(1,2)}|\|_{L_p}\geq\vp$ then $\||x|-|y|\|_{L_p}\geq \vp$ and if $\||x\one_{(1,2)}|-|y\one_{(1,2)}|\|_{L_p}< \vp$ then $\||x|-|y|\|_{L_p}\geq \delta K^{-1}$. 
\\

We now establish \eqref{eq:if-then} with  $\vp=1/8$ and $\delta=1$.  Let $x=\sum a_j(r_j+2^{j/p}\one_{[1+2^{-j},1+2^{-j+1})})$ and $y=\sum b_j(r_j+2^{j/p}\one_{[1+2^{-j},1+2^{-j+1})})$.  We let $f=\sum a_j r_j$ and $g=\sum b_j r_j$ and assume that $\|f\|_{L_2}^2=\sum |a_j|^2=1$, $\|g\|_{L_2}^2=\sum |b_j|^2\leq 1$ and $\langle f,g\rangle=\sum a_j b_j=0$.  We may assume that 
\begin{equation}\label{E:pbound}
(\sum ||a_j|-|b_j||^p)^{1/p}=\||x\one_{(1,2)}|-|y\one_{(1,2)}|\|_{L_p}< \vp=1/8.
\end{equation}
All that remains is to prove that $\||f|^2-|g|^2\|^2_{L_2}\geq \delta$.  
We have from \eqref{E:pbound} that $||a_j|-|b_j||\leq 1/8$ for all $j\in\N$.  Hence, $||a_j|^2-|b_j|^2|\leq 1/4$ for all $j\in\N$ as $|a_j|+|b_j|\leq 2$.  As $r_j^2=\one_{[0,1]}$ for all $j\in\N$,  we have that 

\begin{equation}\label{E:walsh}
    f^2-g^2=(f-g)(f+g)=2\sum_{j>i}(a_ja_i-b_jb_i)r_jr_i+\sum (a_j^2-b_j^2)\one .
\end{equation}
Note that \eqref{E:walsh} gives an expansion for $f^2-g^2$ in terms of the ortho-normal collection of vectors $\{\one_{[0,1]}\}\cup \{r_jr_i\}_{j>i}$.  Thus we have that
\begin{align*}
   2^{-1} &\||f|^2-|g|^2\|_{L_2}^2\geq 2\sum_{j>i}|a_ja_i-b_jb_i|^2\\
    &=\sum_{j\in\N}\sum_{i\in\N}|a_ja_i-b_jb_i|^2-\sum_{j\in\N}|a_j^2-b_j^2|^2\\
    &=\sum_{j\in\N}\Big((\sum_{i\in\N}|a_ja_i|^2+|b_jb_i|^2)-(2a_jb_j\sum_{i\in\N}a_ib_i)\Big)-\sum_{j\in\N}|a_j^2-b_j^2|^2\\
     &=\sum_{j\in\N}(\sum_{i\in\N}|a_ja_i|^2+|b_jb_j|^2)-\sum_{j\in\N}|a_j^2-b_j^2|^2\hspace{1cm}\textrm{ as }\sum a_ib_i=0\\
     &=
     \left(\|f\|_{L_2}^4+\|g\|_{L_2}^4\right)-\sum_{j\in\N}|a_j^2-b_j^2|^2\\
     &\geq 
     \left(\|f\|_{L_2}^4+\|g\|_{L_2}^4\right)-\frac{1}{4}\sum_{j\in\N}|a_j^2-b_j^2|\hspace{1cm}\textrm{ as } |a_j^2-b_j^2|\leq 1/4\\
      &\geq 
     \left(\|f\|_{L_2}^4+\|g\|_{L_2}^4\right)-\frac{1}{4}(\|f\|_{L_2}^2+\|g\|_{L_2}^2)\\
      &= \frac{3}{4}+\|g\|_{L_2}^2 \left(\|g\|_{L_2}^2-\frac{1}{4}\right) \hspace{1cm}\textrm{ as } \|f\|_{L_2}=1\\
      &\geq \frac{3}{4}-\frac{1}{8} \hspace{4.2cm}\textrm{ as } \|g\|_{L_2}\leq1.
\end{align*}
Hence, $\||f|^2-|g|^2\|_{L_2}^2\geq 3/2-\frac{1}{4}> 1=\delta.$
%
\end{proof}

\begin{example}\label{Example not down}
In the special case $p=2$, \Cref{p>2 not down} could have been proven using a result in \cite{calderbank2022stable}. Indeed, as above, let $(r_j)$ denote the Rademacher sequence, realized on the interval $[0,1]$. Define $g_j=r_j+2^{\frac{j}{2}}\one_{[1+2^{-j},1+2^{-j+1})}$. We can think of the sequence $(g_j)$ as being defined on a finite measure space. Note that $\|2^{\frac{j}{2}}\one_{[1+2^{-j},1+2^{-j+1})}\|_{L_2}=1$. Hence, for the same reason as in \cite{calderbank2022stable}, $\overline{\text{span}}\{g_j\}$ does SPR in $L_2$. However, recall that the Rademacher sequence does not  do phase retrieval; we've also scaled the additional indicator functions  to be perturbative in $L_1$. Hence, for $i\neq j$ we have
$\||g_i|-|g_j|\|_{L_1}=\frac{1}{2^i}+\frac{1}{2^j}$, whereas the other side of the SPR inequality is of order $1.$ This provides an example of a subspace $E\subseteq L_2(\mu)\subseteq L_1(\mu)$ that does SPR in $L_2(\mu)$ but not in $L_1(\mu)$. 
\end{example}
As a special case of the next result, we show that for $1\leq q<p<\infty$, if $E$ does SPR in $L_p$ and $L_q$, then we can both interpolate and extrapolate to deduce that $E$ does SPR in $L_r$ for $1\leq r\leq p.$
\begin{theorem}\label{Extrap/interp}
Suppose $\mu$ is a probability measure and $1\leq q<p<\infty$. Let $E$ be a closed subspace of $L_p$ (real or complex). Assume that $\|\cdot\|_{L_p}\sim \|\cdot\|_{L_q}$ on $E$, and $E$ does stable phase retrieval in $L_q$. Then for all $1\leq r\leq p$, $\|\cdot\|_{L_r}\sim \|\cdot \|_{L_p}$ on $E$, and $E$ does stable phase retrieval in $L_r$.
\end{theorem}
\begin{proof}
Assume first that $q<r\leq p$.  Let $C>0$ so that the $L_q$ and $L_p$ norms are $C$-equivalent on $E$, and let $K>0$ so that $E$ does $K$-stable phase retrieval in $L_q$.  As $q<r\leq p$ we have for all $f,g\in E$ that
\begin{equation}
   \inf_{|\lambda|=1} \|f-\lambda g\|_{L_r}\leq  C\inf_{|\lambda|=1} \|f-\lambda g\|_{L_q}\leq CK\||f|-|g|\|_{L_q}\leq CK\||f|-|g|\|_{L_r}.
\end{equation}
Thus, $E$ does stable phase retrieval in $L_r$.
\\

We now turn to the case $1\leq r<q$.  By the previous argument, $E$ does stable phase retrieval in $L_p$.  Hence, the $L_p$ norm is equivalent to the $L_1$ norm on $E$, and hence the $L_p$ norm is equivalent to the $L_r$ norm on $E$. Let $C>0$ so that the $L_p$ and $L_r$ norms are $C$-equivalent on $E$, and let $K>0$ so that $E$ does $K$-stable phase retrieval in $L_p$.
Let $\theta$ be the value so that $q^{-1}=\theta r^{-1}+(1-\theta) p^{-1}$.
By H\"older's inequality, for any $f,g\in E,$
\begin{equation}
    \||f|-|g|\|_{L_q}\leq \||f|-|g|\|_{L_r}^\theta\left(\|f\|_{L_p}+\|g\|_{L_p}\right)^{1-\theta}\leq C\||f|-|g|\|_{L_r}^\theta\left(\|f\|_{L_r}+\|g\|_{L_r}\right)^{1-\theta}.
\end{equation}
Therefore, for any $f,g\in E$, we have 
$$\inf_{|\lambda|=1}\|f-\lambda g\|_{L_r}\leq \inf_{|\lambda|=1}\|f-\lambda g\|_{L_q}\leq K\||f|-|g|\|_{L_q}\leq CK\||f|-|g|\|_{L_r}^\theta\left(\|f\|_{L_r}+\|g\|_{L_r}\right)^{1-\theta}.$$
Thus, $E$ does $\theta$-H\"older stable phase retrieval in $L_r$. By \Cref{HSPR==SPR}, it follows that $E$ does stable phase retrieval in $L_r$.
\end{proof}

In \Cref{p>2 not down}, we showed that when $2\leq p<\infty$ an SPR-subspace $E\subseteq L_p[0,1]$ need not do SPR in $L_q[0,1]$ for any $1\leq q<p$. Our next result shows that the case $1\leq p<2$ is completely different. 

\begin{theorem}\label{p is not always 2}
Let $(\Omega,\mu)$ be a probability space and let $E$ be a closed infinite dimensional subspace of  $L_p(\Omega,\mu)$. Consider the following statements:
\begin{enumerate}
    \item $E$ does stable phase retrieval in $L_p(\Omega,\mu)$.
    \item $E$ does stable phase retrieval in $L_1(\Omega,\mu)$ and $\|\cdot\|_{L_p}\sim \|\cdot\|_{L_1}$ on $E$.
    \item There exists $\alpha>0$ such that for all $x,y\in E$, 
    \begin{equation}
        \mu(\{t\in \Omega: |x(t)|\geq \alpha\|x\|_{L_p} \ \text{and} \ |y(t)|\geq \alpha\|y\|_{L_p}\})>\alpha.
    \end{equation}
\end{enumerate}
Then for all $1\leq p<\infty$, (iii)$\Leftrightarrow$(ii)$\Rightarrow$(i). Moreover, if $1\leq p<2$, all three statements are equivalent.
\end{theorem}

\begin{proof}
$(iii)\Rightarrow (ii)$: Note that condition (iii) implies that $E$ contains no normalized $\alpha^{1+\frac{1}{p}}$-disjoint pairs, when viewed in the $L_p$ norm. Hence, $E$ does SPR in $L_p$, which implies that $\|\cdot\|_{L_p}\sim \|\cdot\|_{L_1}$ on $E$. Using this in condition (iii), we conclude that $E$  contains no normalized almost disjoint pairs, when viewed in the $L_1$ norm, hence does SPR in $L_1$.
\\

$(ii)\Rightarrow (i)$:  Let $C>0$ so that $\|x\|_{L_p}\leq C\|x\|_{L_1}$ for all $x\in E$.  Let $K>0$ so that $E$ does $K$-stable phase retrieval in $L_1$.  Thus, for all $x,y\in E$ we have that
$$ \min_{|\lambda|=1}\|x-\lambda y\|_{L_p}\leq C \min_{|\lambda|=1}\|x-\lambda y\|_{L_1}\leq CK \||x|-|y|\|_{L_1}\leq CK\||x|-|y|\|_{L_p}.
$$
Thus, $E$ does $CK$-stable phase retrieval in $L_p(\Omega)$.
\\

$(i)\Rightarrow(iii)$: Let $1\leq p<2$ and assume that (i) is true but (iii) is false. We first note that condition (i) implies that $\|\cdot\|_{L_1}\sim \|\cdot\|_{L_p}$ on $E$. We may choose a sequence of pairs $(x_n,y_n)_{n=1}^\infty$ in $E$ and $\alpha>0$ such that $\|x_n\|_{L_p}= \|y_n\|_{L_p}=1$, with
\begin{equation}
    \mu(\{t\in\Omega:|x_n|\wedge |y_n|\geq n^{-1}\})\rightarrow 0, \ \text{but} \  \||x_n|\wedge |y_n|\|_{L_p}\geq 2\alpha.
\end{equation}
As $(|x_n|\wedge|y_n|)_{n=1}^\infty$ converges in measure to $0$ and is uniformly bounded below in $L_p$ norm, after passing to a subsequence we may  find a sequence of disjoint subsets $(\Omega_n)_{n=1}^\infty\subseteq\Omega$ such that  
\begin{equation}
    \|(|x_n|\wedge |y_n|) \one_{\Omega_n^c}\|_{L_p}=\||x_n|\wedge |y_n|-(|x_n|\wedge |y_n|) \one_{\Omega_n}\|_{L_p}\to 0.
\end{equation}

Let $\vp_n\searrow0$ with $\vp_1<\alpha/2$.  After passing to a subsequence, we may assume that $\|x_n \one_{\Omega_n}\|_{L_p}\geq \alpha$ for all $n\in\N$.  As $(\Omega_n)_{n=1}^\infty$ is a sequence of disjoint subsets of the probability space $(\Omega,\mu)$, we have that $\mu(\Omega_n)\rightarrow 0$.  Thus, after passing to a further subsequence we may assume that $\|x_j \one_{\Omega_n}\|_{L_p}<\vp_n$ for all $j<n$.  Again, after passing to a further subsequence we may assume that there exists values $(\beta_n)_{n=1}^\infty$ such that $\lim_{j\rightarrow\infty}\|x_j\one_{\Omega_n}\|_{L_p}=\beta_n$ for all $n\in\N$.  Furthermore, we may assume that $\|x_j\one_{\Omega_n}\|_{L_p}<\beta_n+\vp_n/2$ for all $j>n$. As $(\Omega_j)_{j=1}^\infty$ is a sequence of disjoint sets, we have for all $N\in \N$ that
$$\lim_{j\rightarrow\infty}\|x_j\|^p_{L_p}\geq \lim_{j\rightarrow\infty}\sum_{n=1}^N \|x_j\one_{\Omega_n}\|_{L_p}^p=\sum_{n=1}^N \beta_n^p.
$$
In particular, we have that $\beta_n\rightarrow 0$.  Hence, after passing to a further subsequence of $(x_n)_{n=1}^\infty$ we may assume that $\beta_n<\vp_n/2$ for all $n\in\N$.  Thus,
$\|x_j\one_{\Omega_n}\|_{L_p}<\vp_n$ for all $j>n$. In summary, we have that for all $n\in \mathbb{N}$, $\|x_n\one_{\Omega_n}\|_{L_p}\geq \alpha$ and for all $j\neq n$, we have $\|x_j\one_{\Omega_n}\|_{L_p}<\varepsilon_n.$
\\

As $\vp_1<\alpha/2$, we have in particular that $\|x_j-x_n\|_{L_p}\geq \alpha/2$ for all $j\neq n$.  
We have that $(x_n)_{n=1}^\infty$ is a semi-normalized sequence in a closed subspace of $L_p$ which does not contain $\ell_p$.  Thus, by \cite[Theorem 8]{MR312222}, $(x_n)_{n=1}^\infty$ is equivalent to a semi-normalized sequence in $L_{p'}(\nu)$ for some $p<p'\leq 2$ and probability measure $\nu$.
We may assume after passing to a subsequence that $(x_n)_{n=1}^\infty$ is weakly convergent in $L_{p'}(\nu)$.  Thus, the sequence $(x_{2n}-x_{2n-1})_{n=1}^\infty$ converges weakly to $0$ in $L_{p'}(\nu)$. As $L_{p'}(\nu)$ has an unconditional basis, after passing to a further subsequence, we may assume that $(x_{2n}-x_{2n-1})_{n=1}^\infty$ is  $C$-unconditional for some constant $C$.  
\\

As $L_{p'}(\nu)$ has type $p'$ and $(x_{2n}-x_{2n-1})_{n=1}^\infty$ is unconditional, we have that $(x_{2n}-x_{2n-1})_{n=1}^\infty$ is dominated by the unit vector basis of $\ell_{p'}$.  
We will prove that there exists a constant $K$ so that for all $N\in\N$ there exists $k\in\N$ such that the finite sequence $(x_{2n}-x_{2n-1})_{n=k+1}^{k+N}$ $K$-dominates the unit vector basis of $\ell_p^N$. As $p<p'$, this would contradict that $(x_{2n}-x_{2n-1})_{n=1}^\infty$ is dominated by the unit vector basis of $\ell_{p'}$. Alternatively, one could use that $L_{p}$ has type $p$, the  uniform containment of $\ell_p^N$, and  \cite[Theorem 13]{MR312222} to get that $E$ contains a subspace isomorphic to $\ell_p$, which, in view of \Cref{summary}, contradicts that $E$ does stable phase retrieval in $L_p$.\\


Let $N\in\N$ and $\vp>0$.  Let $k\in\N$ be large enough so that $ 2\varepsilon_k N<2^{-1}\alpha$. Let $(a_j)_{j=k+1}^{k+N}$ be a sequence of scalars.  We have that
\begin{align*}
    \|\sum_{j=k+1}^{k+N} &a_j(x_{2j}-x_{2j-1})\|_{L_p}^p 
     \geq\sum_{n=k+1}^{k+N}\|\sum_{j=k+1}^{k+N} a_j(x_{2j}-x_{2j-1})\|_{L_p(\Omega_{2n})}^p\\
     &\geq\sum_{n=k+1}^{k+N}\Big(2^{1-p}\|a_n(x_{2n}-x_{2n-1})\|_{L_p(\Omega_{2n})}^p-\|\sum_{j\neq n} a_j(x_{2j}-x_{2j-1})\|_{L_p(\Omega_{2n})}^p\Big)\\
     &\geq\sum_{n=k+1}^{k+N}\Big(2^{1-p}\alpha^p |a_n|^p-2^p\vp_k^p\left(\sum_{j\neq n}|a_j|\right)^p\Big)\\
     &\geq\sum_{n=k+1}^{k+N}\Big(2^{1-p}\alpha^p |a_n|^p-2^p\vp_k^p N^{p-1}\sum_{j\neq n}|a_j|^p\Big)\\
        &\geq\left( 2^{1-p}\alpha^p-2^p\varepsilon_k^p N^p\ \right) \sum_{n=k+1}^{k+N} |a_n|^p\\
        &\geq 2^{-p}\alpha^p \sum_{n=k+1}^{k+N} |a_n|^p.
\end{align*}

Now that we have established that all three statements in \Cref{p is not always 2} are equivalent for $1\leq p<2$, we can show the implication (ii)$\Rightarrow$(iii) for $1\leq p<\infty$. Indeed, we assume (ii) holds. 
Since $E$ does SPR in $L_1$, by (ii)$\Rightarrow$(iii) for $p=1$, we deduce that there exists $\alpha>0$ such that for all $x,y\in E$, 
    \begin{equation}\label{KP Type SPR}
        \mu(\{t\in \Omega: |x(t)|\geq \alpha\|x\|_{L_1} \ \text{and} \ |y(t)|\geq \alpha\|y\|_{L_1}\})>\alpha.
    \end{equation}
    Now we use the second assumption of (ii) to replace the $L_1$ norm with the $L_p$ norm in \eqref{KP Type SPR}.
\end{proof}

\section{$C(K)$-spaces with SPR subspaces}\label{s:SPR CK}

Throughout this section, subspaces are assumed to be closed and infinite dimensional, unless otherwise mentioned. Recall that a non-empty compact Hausdorff space is called \emph{perfect} if it has no isolated points, and \emph{scattered} (or \emph{dispersed}) if it contains no perfect subsets.
For a compact Hausdorff space $K$, we define its \emph{Cantor-Bendixson derivative} $K'$ to be the set of all non-isolated points of $K$. Clearly $K'$ is closed, and $K = K'$ iff $K$ is perfect; otherwise, $K'$ is a proper subset of $K$. Also, if $K$ contains a perfect set $S$, then $S$ lies inside of $K'$ as well.

\begin{theorem}\label{t:C(K) SPR}
Suppose $K$ is a  compact Hausdorff space. Then $C(K)$ has an SPR subspace if and only if $K'$ is infinite.
\end{theorem}

The proof depends on an auxiliary result, strengthening \Cref{r:emb into C of ball}.

\begin{proposition}\label{p:separble into CK}
Every separable Banach space embeds isometrically into $C(\Delta)$, and into $C[0,1]$, as a $10$-SPR subspace $($here $\Delta$ is the Cantor set$)$.
\end{proposition}

\begin{proof}
Fix a separable Banach space $E$. Let $K$ be the unit ball of $E^*$, with its weak$^*$ topology. By \Cref{l:functionals} and \Cref{r:emb into C of ball}, the natural isometric embedding $j : E \to C(K)$ (taking $e$ into the function $K \to \R : e^* \mapsto e^*(e)$) is such that $\| |jx| \wedge |jy| \| \geq 1/5$ whenever $\|x\| = 1 = \|y\|$. As $K$ is compact and metrizable, there exists a continuous surjection $\Delta \to K$ \cite[Theorem 4.18]{Kechris}; this generates a lattice isometric embedding of $C(K)$ into $C(\Delta)$, hence one can find an isometric copy of $E \subseteq C(\Delta)$ so that $\| |x| \wedge |y| \| \geq 1/5$ whenever $x,y$ are norm one elements of $E$.
\\

View $\Delta$ as a subset of $[0,1]$. Then there exists a positive unital isometric extension operator $T  : C(\Delta) \to C[0,1]$ -- that is, for $f \in C(\Delta)$, $Tf|_\Delta = f$; $T 1 = 1$; $\|T\|=1$; and $Tf \geq 0$ whenever $f \geq 0$. The ``standard'' construction of $T$ involves piecewise-affine extensions of functions from $\Delta$ to $[0,1]$; for a more general approach, see the proof of \cite[Theorem 4.4.4]{AK}. One observes that $\| |Tx| \wedge |Ty|\| \geq \||x| \wedge |y|\|$, hence, if $E \subseteq C(\Delta)$ has the property described in the preceding paragraph, then $\| |Tx| \wedge |Ty|\| \geq 1/5$ whenever $x,y \in E$ have norm $1$.
\\

By \Cref{Thm1}, the copies of $E$ in $C(\Delta)$ and $C[0,1]$ described above do $10$-SPR.
\end{proof}

The next result is standard topological fare (cf.~\cite[Theorem 29.2]{MR3728284}).
\begin{lemma}\label{l:topology}
Suppose $K$ is a  compact Hausdorff space, and $t \in U \subseteq K$, where $U$ is an open set. Then there exists an open set $V$ so that $t \in V \subseteq \overline{V} \subseteq U$.
\end{lemma}

\begin{proof}[Proof of \Cref{t:C(K) SPR}]
Suppose first that $K'$ is finite (in this case, $K$ must be scattered). To show that any subspace $E \subseteq C(K)$ fails SPR, consider $C_0(K,K') = \{f \in C(K) : f|_{K'} = 0\}$. Then $\dim C(K)/C_0(K,K') = |K'| < \infty$, hence $E \cap C_0(K,K')$ is infinite dimensional as well. It suffices therefore to show that every infinite dimensional subspace of $C_0(K,K')$ fails SPR.
\\

Note that, in the case of finite $K'$, $C_0(K,K')$ can be identified with $c_0(K \backslash K')$ as a Banach lattice. Indeed, any $f \in c_0(K \backslash K')$ is continuous on $K \backslash K'$, since this set consists of isolated points only. Extend $f$ to a function $\widetilde{f} : K \to \R$ with $\widetilde{f}|_{K'} = 0$, $\widetilde{f}|_{K \backslash K'} = f$. Note that for any $c>0$, the set $\{t \in K \backslash K' : |f(t)| \geq c\}=\{t\in K : |\widetilde{f}(t)|\geq c\}$ is finite, hence closed; consequently, $\{t \in K : |\widetilde{f}(t)| < c\}$ is an open neighborhood of any element of $K'$. From this it follows that $\widetilde{f}$ is continuous.
\\

On the other hand, pick $h \in C_0(K,K')$. We claim that $h |_{K \backslash K'} \in c_0(K \backslash K')$ -- that is, $\{t \in K \backslash K' : |h(t)| > c\}$ is finite for any $c>0$. Suppose, for the sake of contradiction, that this set is infinite for some $c$. By the compactness of $K$, this set must have an accumulation point, which must lie in $K'$. This, however, contradicts the continuity of $h$.
\\

A ``gliding hump'' argument shows that no subspace of $c_0(K \backslash K')$ does SPR. From this we conclude that no subspace of $C(K)$ does SPR if $K'$ is finite.
\\

Now suppose $K$ contains a perfect set. By \cite[Theorem 2, p.~29]{Lacey74}, there exists a continuous surjection $\phi : K \to [0,1]$. This map generates a lattice isometric embedding $T : C[0,1] \to C(K) : f \mapsto f \circ \phi$. However, $C[0,1]$ contains SPR subspaces, by \Cref{p:separble into CK}.
\\

It remains to prove that $C(K)$ contains an SPR copy of $c_0$ when $K$ is scattered, and $K'$ is infinite. Note first that $K'\backslash K''$ must be infinite. Indeed, otherwise any point of $K'' = K' \backslash (K'\backslash K'')$ will be an accumulation point of the same set, and $K''$ will be perfect, which is impossible. 
\\

Observe also that any $t \in K' \backslash K''$ is an accumulation point of $K \backslash K'$. Indeed, suppose otherwise, for the sake of contradiction. Then $t$ has an open neighborhood $W$, disjoint from $K \backslash K'$. If $U$ is another open neighborhood of $t$, then so is $U \cap W$. As $t$ is an accumulation point of $K$, $U \cap W$ must meet $K$, hence also $K'$. This implies $t \in K''$, providing us with the desired contradiction.
\\

Find distinct points $t_1, t_2, \ldots \in K' \backslash K''$. For each $i$ find an open set $A_i \ni t_i$ so that $t_j \notin A_i$ for $j \neq i$.
\Cref{l:topology} permits us to find an open set $U_i$ so that $t_i \in U_i \subseteq \overline{U_i} \subseteq A_i$. Replacing $U_2$ by $U_2 \backslash \overline{U_1}$, $U_3$ by $U_3 \backslash \overline{U_1 \cup U_2}$, and so on, we can assume that the sets $U_i$ are disjoint. \Cref{l:topology} guarantees the existence of open sets $V_i$ so that, for every $i$, $t_i \in V_i \subseteq \overline{V_i} \subseteq U_i$. 
\\

As noted above, each $t_i$ is an accumulation point of $K \backslash K'$. Therefore, we can find distinct points $(s_{ji})_{j=1}^\infty \subseteq (K \backslash K') \cap V_i$. 
For each $n$, let $S_n$ be the closure of $\{s_{j,2n} : j \in \N\}$ (note $S_n \subseteq \overline{V_{2n}} \subseteq U_{2n}$). 
Note that there exists $x^{(n)} \in C(K)$ such that:
\begin{enumerate}
    \item $0 \leq x^{(n)} \leq 1$ everywhere.
    \item $x^{(n)}|_{S_n} = 1/2$.
    \item $x^{(n)}(s_{1,2n-1}) = 1$.
    \item $x^{(n)}(s_{n,2i}) = 1/2$ for $1 \leq i \leq n-1$.
    \item $x^{(n)} = 0$ on $(K \backslash U_{2n}) \backslash \{s_{1,2n-1}, s_{n,2}, s_{n,4}, \ldots, s_{n,2n-2}\}$.
\end{enumerate}
To construct such an $x^{(n)}$, recall that $s_{1,2n-1}, s_{n,2}, s_{n,4}, \ldots, s_{n,2n-2}$ are isolated points of $K$, hence the function $g$, defined by $g(s_{1,2n-1}) = 1$, $g(s_{n,2i}) = 1/2$ for $1 \leq i \leq n-1$, and $g = 0$ everywhere else, is continuous. Further, by Urysohn's Lemma, there exists $h \in C(K)$ so that $0 \leq h \leq 1/2 = h|_{S_n}$, vanishing outside of $U_{2n}$. Then $x^{(n)} = g+h$ has the desired properties.
\\

We claim that $(x^{(n)})$ is equivalent to the standard $c_0$-basis. Indeed, suppose $(\alpha_n) \in c_{00}$, with $\vee_n|\alpha_n|=1$. We need to show $\|\sum_n \alpha_n x^{(n)}\| = 1$. The lower estimate on the norm is clear, since $x = \sum_n \alpha_n x^{(n)}$ attains the value of $\alpha_n$ at $s_{1,2n-1}$.
\\

For an upper estimate, note that $x$ vanishes outside of $\cup_m U_m$, and on $U_m$ if $m$ is large enough. If $m$ is odd ($m = 2n-1$), then the only point of $U_m$ where $x$ does not vanish is $s_{1,2n-1}$, which we have already discussed. If $m$ is even ($m=2n$), then $|x| \leq 1/2$ except for the points $s_{i,2n}$ ($i > n$); at these points, $x$ equals $(\alpha_n + \alpha_i)/2$, which has absolute value not exceeding $1$. 
\\

It remains to show that $E = \spn[x^{(n)} : n \in \N]$ does SPR. In light of \Cref{Thm1}, if suffices to prove that $\| |x| \wedge |y| \| \geq 1/3$ for any norm one $x, y \in E$. Write $x = \sum_n \alpha_n x^{(n)}$ and $y = \sum_n \beta_n x^{(n)}$. Find $n$ and $m$ so that $|\alpha_n| = 1 = |\beta_m|$. If $n=m$, then both $|x|$ and $|y|$ equal $1$ at $s_{1,2n-1}$, so $\| |x| \wedge |y| \| = 1$. 
\\

Otherwise, assume, by relabeling, that $n<m$. If $|\alpha_m| \geq 1/3$, then 
$$\| |x| \wedge |y| \| \geq |x(s_{1,2m-1})| \wedge |y(s_{1,2m-1})| = |\alpha_m| \wedge|\beta_m| \geq \frac13 .$$ 
The case of $|\beta_n| \geq 1/3$ is treated similarly. If $|\alpha_m|, |\beta_n| < 1/3$, then $|x(s_{m,2n})| = |\alpha_n + \alpha_m|/2 > 1/3$, and similarly, $|y(s_{m,2n})| > 1/3$, which again gives  us $\| |x| \wedge |y| \| \geq 1/3$.
\end{proof}

\begin{question}\label{r:examples of SPR subspaces}
The proof of \Cref{t:C(K) SPR} shows that $K'$ is infinite iff $C(K)$ contains an SPR copy of $c_0$. If $K$ is ``large'' enough (in terms of the smallest ordinal $\alpha$ for which $K^{(\alpha)}$ is finite), what SPR subspaces (other than $c_0$) does $C(K)$ have? Note that $c_0$ is isomorphic to $c = C[0,\omega]$ ($\omega$ is the first infinite ordinal). If $K^{(\alpha)}$ is infinite, does $C(K)$ contain an SPR copy of $C[0,\omega^\alpha]$? This question is of interest even for separable $C(K)$, i.e., metrizable $K$. 
\end{question}

In the spirit of \Cref{rem-clark}, it is natural to ask which (isometric) subspaces of $C(K)$ are necessarily SPR. Below we give a ``very local'' condition on a Banach space $E$ (finite or infinite dimensional) which guarantees that any isometric embedding of $E$ into $C(K)$ has SPR. 
\\

Recall (see \cite{Ja64}) that a Banach space $E$ is called \emph{uniformly non-square} if there exists $\vp > 0$ so that, for any norm one $f,g \in E$ we have $\min\{\|f+g\|, \|f-g\|\} < 2-\varepsilon$. Note that $E$ fails to be uniformly non-square iff for every $\varepsilon > 0$ there exist norm one $f,g \in E$ so that $\|f+g\|, \|f-g\| > 2-\varepsilon$. In the real case, this means that $E$ contains $\ell_1^2$ (equivalently, $\ell_\infty^2$) with arbitrarily small distortion. This is incompatible with uniform convexity or uniform smoothness.

\begin{proposition}\label{p:must be SPR}
Any uniformly non-square subspace of $C(K)$ does SPR.
\end{proposition}

\begin{proof}
Suppose $E$ is a non-SPR subspace of $C(K)$; we shall show that it fails to be uniformly non-square. To this end, fix $\varepsilon \in (0,1/2)$; by \Cref{Thm1}, there exist norm one $f, g \in E$ with $\| |f| \wedge |g| \| < \varepsilon$. Pointwise evaluation shows that
$$
|f| \vee |g| + |f| \wedge |g| \geq |f+g| \geq |f| \vee |g| - |f| \wedge |g| .
$$
As the ambient lattice is an M-space, we have $\| |f| \vee |g| \| = 1$, hence
$$
1 - \vp < \| |f| \vee |g| \| - \| |f| \wedge |g| \| \leq \|f+g\| \leq \| |f| \vee |g| \| + \| |f| \wedge |g| \| < 1 + \vp .
$$
Replacing $g$ by $-g$, we conclude that $1 - \vp < \|f-g\| < 1 + \vp$.
\\

Let $u = (f+g)/\|f+g\|$ and $v = (f-g)/\|f-g\|$. Then
$$
\big\| u - (f+g) \big\| = \big| 1 -\|f+g\| \big| < \varepsilon ,
$$
and similarly, $\big\| v - (f-g) \big\| < \varepsilon$. Then
$$
\|u+v\| \geq \big\| (f+g) + (f-g) \| - \big\| u - (f+g) \big\| - \big\| v - (f-g) \big\| > 2 - 2 \varepsilon ,
$$
and likewise, $\|u-v\| > 2 - 2 \varepsilon$. As $\varepsilon$ is arbitrary, $E$ fails to be uniformly non-square.
\end{proof}

For infinite dimensional subspaces, \Cref{p:must be SPR} is only meaningful when $K$ is not scattered. Indeed, if $K$ is scattered, then $C(K)$ is $c_0$-saturated \cite[Theorem 14.26]{FHHMZ}, hence any infinite dimensional subspace of $C(K)$ contains an almost isometric copy of $c_0$ \cite[Proposition 2.e.3]{LT1}. In particular, such subspaces contain almost isometric copies of $\ell_1^2$, hence they cannot be uniformly non-square.
\\

In light of \Cref{p:must be SPR}, we ask:

\begin{question}
Which Banach spaces $E$ isometrically embed into $C(K)$ in a non-SPR way? 
\end{question}

Note that containing an isometric copy of $\ell_\infty^2$ (and consequently, failing to be uniformly non-square) does not automatically guarantee the existence of a non-SPR embedding into $C(K)$ (in this sense, the converse to \Cref{p:must be SPR} fails). In the following example we look at isometric embeddings only; one can modify this example to allow for sufficiently small distortions.
\\

\begin{proposition}\label{SPR embedding 3d}
There exists a $3$-dimensional space $E$, containing $\ell_\infty^2$ isometrically $($and consequently, failing to be uniformly non-square$)$, so that, if $K$ is a Hausdorff compact, and $J : E \to C(K)$ is an isometric embedding, then $\| |Jx| \wedge |Jy| \| \geq 1/3$ for any norm one $x,y \in E$.
\end{proposition}

The following lemma is needed for the proof, and may be of interest in its own right.

\begin{lemma}\label{l:extreme points}
Suppose $K$ is a Hausdorff compact, $E$ is a Banach space, and $J : E \to C(K)$ is an isometric embedding. Denote by ${\mathcal{F}}$ the set of all extreme points of the unit ball of $E^*$. Then, for any $x, y \in E$, $\| |Jx| \wedge |Jy| \| \geq \sup_{e^* \in {\mathcal{F}}} |e^*(x)| \wedge |e^*(y)|$.
\end{lemma}

\begin{proof}
 Standard duality considerations tell us that $J^* : M(K) \to E^*$ ($M(K)$ stands for the space of Radon measures on $K$) is a strict quotient -- that is, for any $e^* \in E^*$ there exists $\mu \in M(K)$ so that $\|\mu\| = \|e^*\|$ and $J^* \mu = e^*$. Further, we claim that, for any $e^* \in {\mathcal{F}}$, there exists $t \in K$ so that $J^* \delta_t \in \{e^*,-e^*\}$. Indeed, the set $S = \{\mu \in M(K) : \|\mu\| \leq 1, J^* \mu = e^*\}$ is weak$^*$-compact, hence it is the weak$^*$-closure of the convex hull of its extreme points. We claim that any such extreme point is also an extreme point of $\{\mu \in M(K) : \|\mu\| \leq 1\}$. Indeed, suppose $\mu = (\mu_1+\mu_2)/2$, with $\|\mu_1\|, \|\mu_2\| \leq 1$. Then $e^* = (J^* \mu_1 + J^* \mu_2)/2$, which guarantees that $e^* = J^* \mu_1 = J^* \mu_2$, so $\mu_1, \mu_2 \in S$, and therefore, they coincide with $\mu$.
 \\
 
 To finish the proof, recall that the extreme points of $\{\mu \in M(K) : \|\mu\| \leq 1\}$ are point evaluations and their opposites.
\end{proof}

\begin{proof}[Proof of \Cref{SPR embedding 3d}]
To obtain $E$, equip $\R^3$ with the norm
\begin{equation}
\label{eq:norm}
\|(x,y,z)\| = \max\big\{ |x|, |y|, \frac12 \big( |x|+|y|+|z| \big) \big\} .
\end{equation}
Clearly $\{(x_1,x_2,0) : x_1, x_2 \in \R\}$ gives us an isometric copy of $\ell_\infty^2$ in $E$.
Note that the unit ball of $E^*$ is a polyhedron with vertices $(\pm 1, 0, 0)$, $(0, \pm 1, 0)$, and $(\pm 1/2, \pm 1/2, \pm 1/2)$; we denote this set of vertices by ${\mathcal{F}}$. In light of \Cref{l:extreme points}, we have to show that, for any norm one $x = (x_1, x_2, x_3)$ and $y = (y_1, y_2, y_3)$ in $E$, there exists $e^* \in {\mathcal{F}}$ so that $|e^*(x)| \wedge |e^*(y)| \geq 1/3$. 
\\

In searching for $e^*$, we deal with several cases separately. Note first that, if $|x_1| \wedge |y_1| \geq 1/3$, then $e^* = (1,0,0)$ has the desired properties. The case of $|x_2| \wedge |y_2| \geq 1/3$ is treated similarly. Henceforth we assume $|x_1| \wedge |y_1|, |x_2| \wedge |y_2| < 1/3$. In light of \eqref{eq:norm}, we need to consider three cases:

(i) $|x_1| = 1 = |y_2|$ or  $|x_2| = 1 = |y_1|$.

(ii) Either $|x_1| \vee |x_2| = 1$ and $|y_1| + |y_2| + |y_3| = 2$, or $|y_1| \vee |y_2| = 1$ and $|x_1| + |x_2| + |x_3| = 2$.

(iii) $|x_1| + |x_2| + |x_3| = 2 = |y_1| + |y_2| + |y_3|$.

In all the three cases, we look for $e^* = (\vp_1, \vp_2, \vp_3)/2$, with $\vp_1, \vp_2, \vp_3 = \pm 1$ selected appropriately.
\\

{\bf Case (i)}.  We shall assume $x_1 = 1 = y_2$, as other permutations of indices and choices of sign are handled similarly. Select $\vp_1 = 1$, and take $\vp_3$ so that $\vp_3 x_3 \geq 0$. 
Pick $\vp_2 = 1$ if $\vp_1 y_1 + \vp_3 y_3 \geq 0$ and $\vp_2 = -1$ otherwise. Then $|x_2| < 1/3$, hence
$$e^*(x) = \frac12 \big(\vp_1 + \vp_2 x_2 + \vp_3 x_3\big) \geq \frac{1 - |x_2|}2 > \frac{1 - 1/3}2 = \frac13.$$  
Further,
$$|e^*(y)| = \frac{|\vp_1 y_1 + \vp_2 + \vp_3 y_3|}2 \geq \frac12.$$

{\bf Case (ii)}. We deal with $x_1 = 1$ (and consequently, $|y_1| < 1/3$) and $|y_1| + |y_2| + |y_3| = 2$, as other possible settings can be treated similarly. Let $\vp_1 = 1$. If $|x_2| < 1/3$, select $\vp_3$ so that $\vp_3 x_3 \geq 0$. Pick $\vp_2$ so that $\vp_2 y_2$ and $\vp_3 y_3$ have the same sign. Then
$$
|e^*(x)| \geq \frac{1 + |x_3| - |x_2|}2 \geq \frac{1 - |x_2|}2 \geq \frac{1 - 1/3}2 = \frac13 , $$ 
and
$$ |e^*(y)| \geq \frac{|y_2| + |y_3| - |y_1|}2 = \frac{2 - 2|y_1|}2 \geq \frac{2 - 2 \cdot 1/3}2 = \frac23 . $$

Suppose, conversely, that $|x_2| \geq 1/3$, hence $|y_2| < 1/3$. Let $\vp _2 = \sign \, x_2$. Select $\vp_3$ so that $\vp_1 y_1$ and $\vp_3 y_3$ are of the same sign. Then $|x_3| \leq 2 - (1+|x_2|) = 1-|x_2|$, hence
$$
|e^*(x)| \geq \frac{1 + |x_2| - |x_3|}2 \geq \frac{2|x_2|}2 \geq \frac13 . $$ 
On the other hand, $2-|y_2| = |y_1|+|y_3|$ and
$$ |e^*(y)| \geq \frac{|y_1| + |y_3| - |y_2|}2 = \frac{2 - 2|y_2|}2 \geq \frac{2 - 2 \cdot 1/3}2 \geq \frac23 . $$

{\bf Case (iii)}.
If $|x_1|, |x_2| < 1/3$, let $\vp_3 = \sign \, x_3$, and select $\vp_1, \vp_2$ so that both $\vp_1 y_1$ and $\vp_2 y_2$ have the same sign as $\vp_3 y_3$. Then
$$
|e^*(x)| \geq \frac{|x_3| - |x_1| - |x_2|}2 = \frac{2 - 2(|x_1| + |x_2|)}2 \geq \frac{2-4\cdot1/3}2 = \frac13, $$ 
and 
$$ |e^*(y)| = \frac{|y_1| + |y_2| + |y_3|}2 = 1. $$
The case of $|y_1|, |y_2| < 1/3$ is handled similarly.
\\

Now suppose neither of the above holds. Up to a permutation of indices, we assume that  $|x_1| \geq 1/3$ (hence $|y_1| < 1/3$), and $|y_2| \geq 1/3$ (hence $|x_2| < 1/3$).  Then let $\vp_1 = \sign \, x_1$ and $\vp_3 = \sign \, x_3$. Pick $\vp_2$ so that $\sign \, \vp_2 y_2 = \sign \, \vp_3 y_3$, then
$$
|e^*(x)| \geq \frac{|x_1| + |x_3| - |x_2|}2 = \frac{2 - 2|x_2|}2 \geq \frac{2-2\cdot1/3}2 = \frac23 , $$ 
and likewise,
$$ |e^*(y)| \geq \frac{|y_2| + |y_3| - |y_1|}2 \geq \frac23 . \qedhere $$
\end{proof}

\section{Open problems}\label{OP}

We now list some open questions and directions for further research. The reader can find additional questions embedded throughout the paper.

\begin{question}\label{Q1}
(Classification of SPR subspaces): Given a Banach lattice $X$, it is of interest to classify the closed subspaces of $X$ that do SPR. This question, of course, can be interpreted in various ways. Possibly the crudest of these is to classify the closed subspaces of $X$ doing SPR up to Banach space isomorphism. One can then refine this  classification by tracking the optimal SPR and isomorphism constants. On the other hand, one can ask about the ``structure" of the collection of SPR subspaces of $X$. For example, whether  certain natural candidates do SPR, or whether they have a further subspace/perturbation which does SPR. Compare with \cite[Theorem 1.1]{calderbank2022stable}, which, within a restricted class of subspaces of $L_2(\mathbb{R})$, is able to classify those that do SPR.
\end{question}


(Discretization):  Phase retrieval is most often studied in terms of recovering a function $f$ from $|Tf|$ where $T$ is a linear transformation, such as the Fourier transform or Gabor transform.  However, any use of phase retrieval in applications requires sampling at only finitely many points.  
Gabor frames are constructed by sampling the short-time Fourier transform at a lattice; however, any frame constructed by sampling the Gabor transform at an integer lattice cannot do phase retrieval.  There has been significant recent interest in determining which sampling points allow for constructing frames which do phase retrieval \cite{MR4162323,grohsphase,grohs2021stable}.
\\

The problem of sampling continuous frames which do stable phase retrieval to obtain frames which do stable phase retrieval was introduced in \cite{fg22} and was shown to be connected to important integral norm discretization problems in approximation theory (such as in \cite{DPSTT,dai2021universal,kashin2021sampling,LT21}).  In \cite{MR3901674} it is proven that if $(x_t)_{t\in\Omega}$ is a bounded  continuous frame of a separable Hilbert space $H$ then there exists sampling points $(t_j)_{j\in J}$ in $\Omega$ such that $(x_{t_j})_{j\in J}$ is a frame of $H$.  The corresponding quantitative and finite dimensional theorem in \cite{LT21} gives that for each $\beta>0$ there are universal constants $B>A>0$ so that  if $(x_t)_{t\in\Omega}$ is a continuous Parseval frame of an $n$-dimensional Hilbert space $H$ and $\|x_t\|\leq \beta n^{1/2}$ for all $t\in\Omega$ then there exists $m$ on the order of $n$ sampling points $(t_j)_{j=1}^m$ in $\Omega$ so that $(m^{-1/2}x_{t_j})_{j=1}^m$ is a frame of $H$ with lower frame bound $A$ and upper frame bound $B$.  The proof of the above theorem relies on the celebrated solution to the Kadison-Singer Problem and its connection to frame partitioning \cite{MSS,NOU}.  It is natural to consider if this discretization theorem holds as well for stable phase retrieval, and the following question is stated in \cite{fg22}.
\begin{question}\label{Q:disc1}
Let $C,\beta>0$.  Do there exist constants $D,\kappa>0$ so that for all $n\in\N$ there exists $m\leq Dn$ so that the following is true:  Let $H$  be an $n$-dimensional Hilbert space, $(\Omega,\mu)$  a probability space, and $(x_t)_{t\in \Omega}$  a continuous  Parseval frame of $H$ which does $C$-stable phase retrieval such that $\|x_t\|\leq \beta\sqrt{n}$ for all $t\in \Omega$.  Then there exists a sequence of sampling points $(t_j)_{j=1}^m\subseteq\Omega$ such that $(m^{-1/2}x_{t_j})_{j=1}^m$ is a frame of $H$ which does $\kappa$-stable phase retrieval.
\end{question}

  Note that if $(x_t)_{t\in \Omega}$ is a continuous  Parseval frame of $H$ over a probability space $\Omega$ then the analysis operator $\Theta(x)=(\langle x,x_t\rangle)_{t\in\Omega}$ is an isometric embedding of $H$ into $L_2(\Omega)$. We have that the continuous frame  $(x_t)_{t\in \Omega}$ does $C$-stable phase retrieval if and only if the range of the analysis operator $\Theta(H)$ does $C$-stable phase retrieval as a subspace of $L_2(\Omega)$.  
 In \Cref{p>2 not down} we prove that there exists a subspace $E\subseteq L_2(\Omega)$ such that $E$ does stable phase retrieval as a subspace of $L_2(\Omega)$ but that $E$ does not do stable phase retrieval as a subspace of $L_1(\Omega)$.  As shown by \Cref{p is not always 2}, doing stable phase retrieval in $L_1(\Omega)$ gives a lot of useful additional structure.  This motivates the following problem.

\begin{question}\label{Q:disc1}
Let $C,\beta>0$.  Do there exist constants $D,\kappa>0$ so that for all $n\in\N$ there exists $m\leq Dn$ so that the following is true: Let  $H$ be an $n$-dimensional Hilbert space, $(\Omega,\mu)$  a probability space, and $(x_t)_{t\in \Omega}$  a continuous  Parseval frame of $H$ with analysis operator $\Theta$ such that $\Theta(H)$ does $C$-stable phase retrieval as a subspace of $L_1(\Omega)$ and as a subspace of $L_2(\Omega)$, and $\|x_t\|\leq \beta\sqrt{n}$ for all $t\in \Omega$.  Then there exists  a sequence of sampling points $(t_j)_{j=1}^m\subseteq\Omega$ such that $(m^{-1/2}x_{t_j})_{j=1}^m$ is a frame of $H$ which does $\kappa$-stable phase retrieval.
\end{question}

The previous two questions on constructing frames by sampling continuous frames relate to discretizing the $L_2$-norm on a subspace of $L_2(\Omega)$.  There is significant interest in approximation theory on discretizing the $L_p$-norm on finite dimensional subspaces of $L_p(\Omega)$ which are called Marcinkiewicz-type discretization problems \cite{DPSTT,dai2021universal,kashin2021sampling,LT21}.  For $p\neq 2$, it is too much to ask for the number of sampling points to be on the order of the dimension of the subspace.  This leads to the following general problem on discretizing stable phase retrieval.

\begin{question}
Let $E\subseteq L_p(\Omega)$ be an $n$-dimensional subspace for some $1\leq p<\infty$ and probability space $\Omega$.  Let $C>0$ and let $f:\N\rightarrow\N$ be strictly increasing.  What properties on $E$ imply that there exists $m\leq f(n)$ and sampling points $(t_j)_{j=1}^m\subseteq\Omega$ so that the subspace $\{(m^{-1/p} x(t_j))_{j=1}^m:x\in E\}$  does $C$-stable phase retrieval in $\ell_p^m$?
 What properties on $E$ imply that there exists $m\leq f(n)$, sampling points $(t_j)_{j=1}^m\subseteq\Omega$, and weights $(w_j)_{j=1}^m$ with $\sum_{j=1}^m |w_j|^p=1$
 so that the subspace $\{(w_j x(t_j))_{j=1}^m:x\in E\}$  does $C$-stable phase retrieval in $\ell_p^m$?

\end{question}

\bibliographystyle{plain}
\bibliography{refs.bib}

\end{document}